\documentclass[12pt,letter,reqno]{amsart} 

\usepackage[text={420pt,660pt},centering]{geometry}
\usepackage{amssymb,amsfonts,amsthm,mathrsfs}
\usepackage{marginnote}
\usepackage{comment}
\usepackage{enumitem}
\usepackage{graphicx}
\usepackage{bm}

\usepackage[dvipsnames]{color}

\usepackage[colorlinks=true, pdfstartview=FitV, linkcolor=black, citecolor=blue, urlcolor=blue]{hyperref}

\def\Xint#1{\mathchoice
{\XXint\displaystyle\textstyle{#1}}%
{\XXint\textstyle\scriptstyle{#1}}%
{\XXint\scriptstyle\scriptscriptstyle{#1}}%
{\XXint\scriptscriptstyle%
\scriptscriptstyle{#1}}%
\!\int}
\def\XXint#1#2#3{{\setbox0=\hbox{$#1{#2#3}{%
\int}$ }
\vcenter{\hbox{$#2#3$ }}\kern-.6\wd0}}
\def\barint{\, \Xint -} 
\def\bariint{\barint_{} \kern-.4em \barint}
\def\bariiint{\bariint_{} \kern-.4em \barint}
\renewcommand{\iint}{\int_{}\kern-.34em \int} 
\renewcommand{\iiint}{\iint_{}\kern-.34em \int} 

\DeclareMathAlphabet{\mathcal}{OMS}{cmsy}{m}{n}

\theoremstyle{plain}

\newtheorem{theorem}{Theorem}[section]

\newtheorem{definition}[theorem]{Definition}
\newtheorem{lemma}[theorem]{Lemma}

\newtheorem{corollary}[theorem]{Corollary}
\newtheorem{proposition}[theorem]{Proposition}

\theoremstyle{definition}
\newtheorem{remark}[theorem]{Remark}

\newcommand{\R}{\mathbb{R}}
\newcommand{\C}{\mathbb{C}}
\newcommand{\N}{\mathbb{N}}
\newcommand{\Z}{\mathbb{Z}}
\newcommand{\T}{\mathbb{T}}
\newcommand{\bP}{\mathbb{P}}

\newcommand{\supp}{\mathop{\mathrm{supp}}}

\newcommand{\be}{\beta}

\newcommand{\p}{\partial}
\newcommand{\la}{\langle}
\newcommand{\ra}{\rangle}
\newcommand{\les}{\lesssim}

\newcommand{\norm}[1]{\lVert #1 \rVert}
\renewcommand{\:}{\colon}

\let\div\relax
\DeclareMathOperator{\div}{div}

\DeclareMathOperator{\curl}{curl}

\let\tilde\relax
\newcommand{\tilde}[1]{\widetilde{#1}}

\newcommand{\I}{{\rm I}}
\newcommand{\II}{{\rm II}}

\newcommand{\eps}{\ensuremath{\varepsilon}}

\newcommand{\da}[1]{{\color{green}{#1}}}

\renewcommand{\L}{\bm{L}}
\renewcommand{\T}{\bm{T}}
\newcommand{\M}{\bm{M}}
\newcommand{\D}{\bm{D}}
\newcommand{\A}{\bm{A}}
\renewcommand{\S}{\bm{S}}
\newcommand{\K}{\bm{K}}
\newcommand{\BS}{{\rm BS}}
\let\Re\relax
\DeclareMathOperator{\Re}{Re}

\let\Pr\relax
\DeclareMathOperator{\Pr}{Pr}

\numberwithin{equation}{section}
\setlist[enumerate]{leftmargin=*}

\title[Non-uniqueness of Leray solutions]{Non-uniqueness of Leray solutions of the forced Navier-Stokes equations}
\author[Albritton]{Dallas Albritton} 
\address[Dallas Albritton]{School of Mathematics, Institute for Advanced Study, 1 Einstein Dr., Princeton, NJ 08540, USA}
\email{dallas.albritton@ias.edu}

\author[Bru\'e]{Elia Bru\'e} 
\address[Elia Bru\'e]{School of Mathematics, Institute for Advanced Study, 1 Einstein Dr., Princeton, NJ 08540, USA}
\email{elia.brue@ias.edu}

\author[Colombo]{Maria Colombo}
\address[Maria Colombo]{Institute of Mathematics, EPFL SB, Station 8,  CH-1015 Lausanne, Switzerland }
\email{maria.colombo@epfl.ch}

\begin{document}
\begin{abstract}

In the seminal work~\cite{leray}, Leray demonstrated the existence of global weak solutions to the Navier-Stokes equations in three dimensions.
We exhibit two distinct Leray solutions with zero initial velocity and identical body force.
Our 
{approach is to construct} a `background' solution which is unstable for the Navier-Stokes dynamics in similarity variables; its similarity profile is a smooth, compactly supported vortex ring whose cross-section is a modification of the unstable two-dimensional vortex constructed by Vishik in~\cite{Vishik1,Vishik2}. The second solution is a trajectory on the unstable manifold associated to the background solution, in accordance with the predictions  of Jia and {\v S}ver{\'a}k in~\cite{jiasverakselfsim,jiasverakillposed}. Our solutions live precisely on the borderline of the known well-posedness theory.

\end{abstract}

\maketitle

\setcounter{tocdepth}{1}
\tableofcontents

\parskip   2pt plus 0.5pt minus 0.5pt

\section{Introduction}


In the seminal work~\cite{leray}, Leray demonstrated the existence of global weak solutions to the three-dimensional Navier-Stokes equations in the whole space:
\begin{equation}
\label{eq:ns}
\tag{NS}
\begin{aligned}
    \p_t u + u \cdot \nabla u - \Delta u + \nabla p &= f \\
    \div u &= 0 \, .
\end{aligned}
\end{equation}
{Leray's work, and in particular his compactness method, has become paradigmatic for the construction of global weak solutions to many nonlinear PDEs.}
His weak solutions, which he called \emph{solutions turbulent}, are now known as \emph{Leray--Hopf solutions}, recognizing also the contribution~\cite{hopf} of Hopf in bounded domains.

\begin{definition} Let $T>0$, $u_0 \in L^2(\R^3)$ be a divergence-free vector field, and $f \in L^1_t L^2_x(\R^3 \times (0,T))$.
A \emph{Leray-Hopf solution} on $\R^3 \times (0,T)$ with initial datum $u_0$ and force $f$ is a divergence-free vector field
\begin{equation}
    \label{eq:minimumregularity}
u \in L^\infty_t L^2_x \cap L^2_t \dot H^1_x(\R^3 \times (0,T))
\end{equation}
 which
\begin{enumerate}

\item  belongs to $C_{\rm w}([0,T];L^2(\R^3))$ and attains the initial data: $u(\cdot,0) = u_0$;
\item solves the Navier-Stokes equations~\eqref{eq:ns} in the sense of distributions on $\R^3 \times (0,T)$  for some pressure $p \in L^1_{\rm loc}(\R^3 \times (0,T))$; and

\item satisfies the \emph{energy inequality} for all $t \in (0,T]$:
\begin{equation}
    \label{eq:energyinequality}
    \frac{1}{2} \int |u(x,t)|^2 \, dx + \int_0^t \int |\nabla u|^2 \, dx \, ds \leq \frac{1}{2} \int |u_0(x)|^2 \, dx + \int_0^t \int f \cdot u \, dx \,ds \, .
\end{equation}
 
\end{enumerate}
\end{definition}


Leray's construction~\cite{leray}  shows that for each divergence-free $u_0 \in L^2(\R^2)$ and $f \in L^1_t L^2_x + L^2_t \dot H^{-1}_x(\R^3 \times \R_+)$, there exists a \emph{global-in-time} Leray--Hopf solution to~\eqref{eq:ns} with initial data $u_0$ and body force $f$. Solutions further satisfy $u(\cdot,t) \to u_0$ as ${t \to 0^+}$ strongly in $L^2(\R^3)$. Their pressure is determined, up to a constant-in-space function of time $c(t)$, by solving 
\begin{equation}
 { - \Delta   p } = \div \div (u \otimes u) - \div f .
\end{equation}
Furthermore, solutions built via Leray's construction are \emph{suitable} {(see  \cite[Proposition 30.1]{lemarie2002})}, namely, they satisfy the \emph{local energy inequality}
\begin{equation}\label{eqn:suitable}
    (\p_t - \Delta) \frac{1}{2} |u|^2 + |\nabla u|^2 + \div \left[ \left( \frac{1}{2} |u|^2 + p \right) u \right]\leq f \cdot u
\end{equation}
distributionally on $\R^3 \times \R_+$. Suitability was crucial in developing dimensional bounds on the potential singular set as done by Caffarelli, Kohn, and Nirenberg~\cite{ckn}. 

\emph{Do Leray--Hopf solutions, perhaps augmented with suitability, constitute an existence and uniqueness class for the Navier--Stokes equations?} Since Leray's work, this has remained an important open problem in the theory of the Navier-Stokes equations, as discussed in~\cite{hopf,ladyzhenskayanonuniqueness} and, in the past decade,~\cite{jiasverakillposed} and \cite[Problem  9]{convexintegrationconstructionsinturbulence}. 
In~1969, Ladyzhenskaya~\cite{ladyzhenskayanonuniqueness} already gave an example of non-uniqueness to~\eqref{eq:ns}, though in a time-varying domain which degenerates as $t \to 0^+$ and with non-standard boundary conditions and a force. In recent years, this problem has been revisited, and various works have influenced the community toward the expectation that uniqueness does not hold. 
We mention two of particular importance:
First, Jia, {\v S}ver{\'a}k, and Guillod~\cite{jiasverakselfsim,jiasverakillposed,guillodsverak} developed a program towards non-uniqueness and supported the missing steps with numerical evidence. Second, Buckmaster and Vicol~\cite{BuckmasterVicolAnnals} constructed non-unique distributional solutions of the Navier-Stokes equations with finite \emph{kinetic} energy via the powerful method of convex integration. Despite several attempts and results obtained alongside these ideas,  convex integration is at the present time far from reaching the minimum regularity~\eqref{eq:minimumregularity} required to speak of Leray-Hopf solutions.

In this paper, we answer the uniqueness question in the negative:

\begin{theorem}[Non-uniqueness]
    \label{thm:introthm}
There exist $T>0$, $f \in L^1_t L^2_x(\R^3_+ \times (0,T))$, and two distinct suitable Leray--Hopf solutions $u$, $\bar{u}$ to the Navier--Stokes equations on $\R^3 \times (0,T)$ with body force $f$  and initial condition $u_0 \equiv 0$.
\end{theorem}
 
The solutions exhibited in this theorem {are smooth for positive times} and satisfy many additional properties, which we describe in Theorem~\ref{thm:refined}.
 It is elementary to extend the distinct solutions to global-in-time Leray--Hopf solutions by modifying the force after $T$ and using Leray's result.  

\subsection{Strategy}

The Navier--Stokes equations possess a scaling symmetry
\begin{equation}
    \label{eq:nsscaling}
    u_\lambda(x,t) = \lambda u(\lambda x,\lambda^2 t), \quad p_\lambda(x,t) = \lambda^2 p(\lambda x, \lambda^2 t), \quad f_\lambda(x,t) = \lambda^3 f(\lambda x, \lambda^2 t),
\end{equation}
arising from dimensional analysis:
\begin{equation}
    [x] = L, \quad [t] = L^2, \quad [u] = L^{-1},\quad [p] = L^{-2}, \quad [f] = L^{-3},
\end{equation}
as discussed in~\cite{ckn}.\footnote{These particular dimensions arise from treating the viscosity $\nu$, which has dimensions $[\nu] = L^2/T$, as a dimensionless quantity.} The scaling symmetry is intimately connected with the well-posedness theory; the \emph{critical norms} (those invariant under the scaling symmetry) mark the borderline between the regime in which the non-linearity dominates (\emph{supercritical}) and the regime in which it can be treated perturbatively (\emph{subcritical}). On the borderline itself, the notion of `size' becomes important, as can be seen from the following heuristic: When $u_0 \sim a/|x|$, we have $\Delta u_0 \sim a/|x|^3$ and $u_0 \cdot \nabla u_0 \sim a^2/|x|^3 $. When $a \ll 1$, $\Delta u_0$ dominates, and we expect well-posedness; when $a \gg 1$, the non-linearity dominates, and we look for non-uniqueness.

A particular class of solutions which lives precisely on the borderline of the known well-posedness theory are the \emph{self-similar solutions}, that is, Navier--Stokes solutions on $\R^3 \times \R_+$ invariant under the scaling symmetry~\eqref{eq:nsscaling}. {These solutions also play an important role in the program of Jia, {\v S}ver{\'a}k, and Guillod~\cite{jiasverakselfsim,jiasverakillposed,guillodsverak}, which we describe further in Section~\ref{sec:comparisonwithexisting}.}

We introduce the \emph{similarity variables}\footnote{We regard $\tau$ as $\log t/t_0$ with $t_0 = 1$; the argument of the logarithm is non-dimensionalized by the reference time $t_0$.}
\begin{equation}
\xi = \frac{x}{\sqrt{t}}, \quad \tau = \log t,
\end{equation}
\begin{equation}
    \label{eq:definitionofuandU}
    u(x,t) = \frac{1}{\sqrt{t}} U(\xi,\tau), \quad f(x,t) = \frac{1}{t^{\frac{3}{2}}} F(\xi,\tau).
\end{equation}
Notice that $\tau \in \R$ whereas $t \in (0,+\infty)$. Our convention is that lower case functions denote functions in physical variables, and upper case functions denote the corresponding functions in similarity variables. In these variables, the Navier--Stokes equations become
\begin{equation}
    \label{eq:similaritynavierstokes}
\begin{aligned}
    \p_\tau U - \frac{1}{2} \left( 1 + \xi \cdot \nabla_\xi \right) U  - \Delta U + U \cdot \nabla U + \nabla P &= F \\
    \div U &= 0 \, .
    \end{aligned}
\end{equation}
A self-similar solution $\bar{u}$ is precisely a steady state $\bar{U}$ of the renormalized flow~\eqref{eq:similaritynavierstokes}.

Suppose furthermore that $\bar{U}$ is \emph{linearly unstable} for the dynamics of~\eqref{eq:similaritynavierstokes}, namely, there exists an unstable eigenvalue for the linearized operator $\L_{\rm ss}$, defined by
\begin{equation}
\label{eq:linearizedsimilaritynavierstokes}
    -\L_{\rm ss} = - \frac{1}{2} \left( 1 + \xi \cdot \nabla_\xi \right) U  - \Delta U + \bP \left( \bar{U} \cdot \nabla U + U \cdot \nabla{\bar{U}} \right) \, ,
\end{equation}
where $\bP$ is the Leray projector. 
 We seek the \emph{unstable manifold} associated to the most unstable eigenvalues, whose real part is denoted by $a$. Solutions on this manifold satisfy the asymptotics
\begin{equation}
    \label{eq:satisfytheasymptotics}
    U = \bar{U} + U^{\rm lin} + O(e^{2 \tau a }),
\end{equation}
as $\tau \to -\infty$, 
where $U^{\rm lin}$ is a non-trivial solution of the linearized equations $\p_\tau U^{\rm lin} = \L_{\rm ss} U^{\rm lin}$ on $\R^3 \times \R$ and corresponding to an unstable eigenfunction. Observing that $U^{\rm lin}$ decays at the rate $e^{\tau a}$, the solution $\bar{U}$ may be regarded as `non-unique at $\tau = -\infty$,' which corresponds to non-uniqueness at $t = 0$ in physical variables.



For our purposes, it will be enough to construct a single trajectory on the unstable manifold. This part of the argument, which we carry out in Section~\ref{sec:nonlinearinstability}, is analogous to that of 
\cite[Theorem 4.1]{jiasverakillposed}, though with  technical differences. For instance, no truncation to achieve finite energy will be necessary here.





Before proceeding, we make two crucial observations: First, motivated by the force in Vishik's work, we allow ourselves a force in~\eqref{eq:similaritynavierstokes}. Then any smooth function of space may be considered a steady solution of the PDE. Second, it is enough to find an unstable steady state for the Euler equations. Heuristically, by increasing the size of the background $\bar{U}$, the main terms in the linearized operator~\eqref{eq:linearizedsimilaritynavierstokes} are those arising from the nonlinearity of the equation; the extra terms, including the Laplacian, can be considered perturbatively. We make this reduction rigorous in Section~\ref{sec:eulertonavierstokes}.

Our central task is therefore to \emph{find a smooth and decaying unstable steady state of the forced Euler equations in three dimensions}. This task is far from elementary. First of all, there is no general tool to construct unstable solutions. Rather, unstable solutions are typically sought among explicit solutions with many symmetries and subsequently analyzed by ODE techniques. In three dimensions, the most natural explicit steady states are \emph{shear flows} and \emph{vortex columns}, which do not decay and hence have infinite energy. {Remarkably, to our knowledge, there is no existing unstable solution in three dimensions suitable for our purposes, with or without forcing.} 

In two dimensions, a natural class of explicit steady states is comprised of \emph{vortices}:
\begin{equation}
\bar{u} = \bar{u}^\varphi(\varrho) e_\varphi,    
\end{equation}
 written here in polar coordinates $(\varrho,\varphi)$. In~\cite{Vishik2}, Vishik constructed unstable vortices with $\bar{\omega}$ having power-law decay as $\varrho \to +\infty$, see Theorem~\ref{thm:vishiksunstableprofile} below.\footnote{At the time of our writing, Vishik's papers~\cite{Vishik1,Vishik2} remain unpublished. A careful exposition of his construction of unstable vortices, with some technical deviations and simplifications, is contained in~\cite{OurLectureNotes} by the present authors and De Lellis, Giri, Janisch, and Kwon.} Interestingly, the motivation for Vishik's construction originates in the classical work of Tollmien~\cite{Tollmien} on shear flows (see the summary in~\cite{DrazinReid}). The rigorous construction of unstable shear flows is due to~\cite{LinSIMA2003}.
 

Our approach is to construct a \emph{vortex ring} which `lifts' Vishik's unstable vortex to three dimensions. The vortex ring we are looking for is an axisymmetric Euler solution without swirl. Recall that axisymmetric-no-swirl velocity fields are of the form $u = u^r(r,z) e_r + u^z(r,z) e_z$. The corresponding vorticity $\omega$ and stream function $\psi$ satisfying $-\Delta \psi = \omega$ are pure swirl: $\omega = - \omega^\theta(r,z) e_\theta$ and $\psi = \psi^\theta(r,z) e_\theta$.\footnote{We impose the sign convention $\omega^\theta = - \omega \cdot e_\theta$ so that $\omega^\theta$ in the axisymmetric-no-swirl setting acts like $\omega$ ($= \omega^z$) does in the two-dimensional setting. Indeed, we have that $e_x \times e_y = e_z$ whereas $e_r \times e_z = -e_\theta$.} In these variables, the Euler equations become
\begin{equation}
    \label{eq:eulerintro}
    \p_t \omega^\theta + u \cdot \nabla \omega^\theta  - \frac{u^r}{r} \omega^\theta = 0
\end{equation}
\begin{equation}
    \label{eq:psithetaequation1}
    \left( \p_r^2  + \frac{1}{r} \p_r  - \frac{1}{r^2}  + \p_z^2 \right) \psi^\theta = \omega^\theta
\end{equation}
\begin{equation}
    u = -\p_z \psi^\theta e_r + \left( \p_r + \frac{1}{r} \right) \psi^\theta e_z.
\end{equation}
Here is the key observation: \emph{At large distances $r \to +\infty$, the axisymmetric Euler equations without swirl formally converge to the two-dimensional Euler equations}
\begin{equation}
    \p_t \omega + u \cdot \nabla \omega = 0
\end{equation}
\begin{equation}
    \label{eq:psithetaequation2}
    \Delta_{x,y} \psi = \omega, \quad u = \nabla^\perp \psi
\end{equation}
with the axisymmetric variables $(r,z)$ corresponding to the two-dimensional variables $(x,y)$.\footnote{The vector Laplacian satisfies $\Delta = \nabla \div + \nabla^\perp \curl$ in two dimensions and $\Delta = \nabla \div - \curl \curl$ in three dimensions. This is a manifestation of the definition (up to a sign convention on the Laplacian) $\Delta = dd^* + d^*d$ on differential $k$-forms. The signs in the equations~\eqref{eq:psithetaequation1} and~\eqref{eq:psithetaequation2} match under the convention that $\omega^\theta = - \omega \cdot e_\theta$.} 
To exploit the key observation, \emph{we put Vishik's unstable vortex into the axisymmetric variables $(r,z)$ and center it at a large distance $r = \ell$.}    This construction corresponds to a `very long' vortex ring. To this end, we first construct a compactly supported unstable vortex in Proposition~\ref{pro:truncatedunstablevortex} by truncating Vishik's vortex. 
We demonstrate in  Proposition~\ref{pro:axisymmetricinstab} that, 
for large enough $\ell$, the resulting vortex ring, whose cross-sectional profile is this new unstable profile, is also unstable.





The above arguments culminate in the following refined version of Theorem~\ref{thm:introthm}.
\begin{theorem}[Non-uniqueness, refined]
    \label{thm:refined}
There exists a smooth, compactly supported velocity profile $\bar{U}$ and a smooth, compactly supported force profile
 \begin{equation}
     \bar{F} := -\frac{1}{2} (1+ \xi \cdot \nabla_\xi) \bar{U} - \Delta \bar{U} + \bar{U} \cdot \nabla \bar{U}
 \end{equation}
satisfying the following properties:
\begin{enumerate}[label=(\Alph*)]
\item\label{item:maintheorema} The linearized operator $\L_{\rm ss}$ defined in~\eqref{eq:linearizedsimilaritynavierstokes} has an unstable eigenvalue $\lambda$ with non-trivial smooth eigenfunction $\eta$ belonging to $H^k(\R^3)$ for all $k \geq 0$: 
\begin{equation}
    \L_{\rm ss} \eta = \lambda \eta \quad \text{ and } \quad a := \Re \lambda > 0.
\end{equation}
 Consider the solution
\begin{equation}
U^{\rm lin}(\cdot,\tau) = \Re ( e^{\tau \lambda} \eta )
\end{equation}
of the linearized equation  $\p_\tau U^{\rm lin} = \L_{\rm ss} U^{\rm lin}$ in $\R^3 \times \R$.
\item\label{item:maintheoremb} There exists $T \in \R$ and a velocity field $U^{\rm per} : \R^3 \times (-\infty,T] \to \R^3$ satisfying
\begin{equation}
\| U^{\rm per}(\cdot,\tau) \|_{H^k} \les_k e^{2 \tau a} \quad \forall \tau \in (-\infty,T]
\end{equation}
for all $k \geq 0$, and, in the notational convention~\eqref{eq:definitionofuandU} above,
\begin{equation}
    \bar{u}, \quad u = \bar{u} + u^{\rm lin} + u^{\rm per}
\end{equation}
are two distinct suitable Leray--Hopf solutions of the Navier--Stokes equations~\eqref{eq:ns} on $\R^3 \times (0,e^T)$ with initial data $u_0 \equiv 0$ and forcing term $\bar{f}$.
\end{enumerate}
\end{theorem}

The solutions constructed above live at critical regularity and therefore satisfy~\eqref{eq:energyinequality} and~\eqref{eqn:suitable} with equality. 
One may easily verify that for any $p \in [2,+\infty]$, $k \geq 0$, and $t \in (0,e^T)$, we have
    \begin{equation}\label{eq: tilde u}
		t^{\frac{k}{2}} \| \nabla^k \bar  u(\cdot, t) \|_{L^p} + t^{\frac{k}{2}} \| \nabla^k  u(\cdot, t)  \|_{L^p} \les_k t^{\frac{1}{2} (\frac{3}{p}-1)} \, ,
	\end{equation}
     \begin{equation}\label{eq:Force Lp}
    	t^{\frac{k}{2}} \| \nabla^k f(\cdot, t)\|_{L^p} 
    	\les_k t^{\frac{1}{2} (\frac{3}{p} - 3)}
    	 \, .
    \end{equation}
    
Our solutions moreover do not break the axisymmetric without swirl structure. However, it is likely that our instability does break the following \emph{nearly} symmetric structure inherited from Vishik's vortex.
Namely, 
in two dimensions, Vishik's instability breaks radial symmetry, retaining only strict $m$-fold rotational symmetry (see \eqref{def:m-fold} below). The cross section of our vortex ring is nearly radial, and although we do not prove it, our unstable direction is likely nearly $m$-fold symmetric but not nearly radial. It is further striking that non-uniqueness holds already in this class because initially smooth axisymmetric solutions without swirl are known to remain globally smooth~\cite{ladyzhenskayaaxisymmetric}. This suggests that non-uniqueness and finite-time singularity are separate issues, though in models exhibiting both, the two issues may be related through the problem of continuation past the singularity.

\subsection{Comparison with existing literature}
    \label{sec:comparisonwithexisting}
We now review the above cited literature more in detail and point out certain connections with the present work.

\emph{Self-similarity, instability, and non-uniqueness in the Navier-Stokes equations: Jia, {\v S}ver{\'a}k, and Guillod's program~\cite{jiasverakselfsim,jiasverakillposed,guillodsverak}}.  Let $a_0 \in C^\infty(\R^3 \setminus \{ 0 \})$ be divergence free and scaling invariant and $\sigma \in \R$ be a size parameter. One may ask whether there exists a scaling invariant solution $u_\sigma$ with initial data $u_{0,\sigma} = \sigma a_0$. For $|\sigma| \ll 1$, existence and uniqueness falls into the known perturbation theory of Koch and Tataru~\cite{kochtataru} in ${\rm BMO}^{-1}$, for example. For arbitrary size, existence was demonstrated in~\cite{jiasverakselfsim} via 
a new local-in-space smoothing estimate and Leray--Schauder degree theory (see~\cite{bradshawtsaiII} and the references it cites for a simplified approach). Thus, the self-similar solutions constructed in~\cite{jiasverakselfsim} with initial data $u_{0,\sigma}$ constitute a continuous-in-$\sigma$ branch connected to the origin, and  non-uniqueness in the Cauchy problem for~\eqref{eq:ns} is conjectured in~\cite{jiasverakillposed} to arise due to bifurcations within or from this class of solutions.\footnote{One may draw analogy between steady states of~\eqref{eq:ns} and~\eqref{eq:similaritynavierstokes}. Bifurcations among steady solutions of~\eqref{eq:ns} (and, hence, non-uniqueness for certain boundary data) occur naturally, as in the emergence of Taylor vortices in Taylor-Couette flow.}
 In~\cite{jiasverakillposed}, spectral conditions on the linearized operator $\L_{\rm ss}^{(\sigma)}$ (that is, linearized around the profiles $U_{\sigma}$) are identified for such bifurcations to occur. Moreover, it is shown how to truncate the solutions to have finite energy without introducing a force (for this, it is roughly necessary to `watch' the eigenvalue become unstable). One of the spectral conditions was verified numerically 
in~\cite{guillodsverak} on certain axisymmetric examples with pure swirl initial datum. 

One of the main difficulties in rigorous verification of the spectral condition in~\cite{jiasverakillposed} is that the self-similar solutions are non-explicit. 
In our work, the inclusion of a force gives us the freedom to search for a more explicit unstable profile. 



\emph{Self-similarity, instability, and non-uniqueness in the Euler equations}. As in the above mentioned works and the present work, self-similarity and instability lie at the core of two recent results on the sharpness of the Yudovich class in the two-dimensional Euler equations, that is, non-uniqueness of solutions with vorticity in $L^p$, $p < +\infty$.  Vishik~\cite{Vishik1} obtains non-uniqueness with a force  belonging to $L^1_t L^p_x$ in the right-hand side of the vorticity equation. His approach is based on the construction~\cite{Vishik2} of an unstable vortex which we use in an essential way (see also the exposition~\cite{OurLectureNotes} by the authors and four others).

Bressan, Murray, and Shen~\cite{BressanSelfSimilar,BressanAposteriori} proposed an initial datum which is expected to generate non-unique Euler solutions without a force. Their mechanism of instability is different from Vishik's and partially inspired by the vortex spirals of Elling~\cite{EllingAlgebraicSpiral,EllingSelfSimilar}. Akin to \cite{jiasverakillposed,guillodsverak}, it is based on a combination of analytical and numerical evidence which is highly suggestive, though not yet known to rise to the level of computer-assisted proof.

\emph{Convex integration}.
Buckmaster and Vicol \cite{BuckmasterVicolAnnals} (see also \cite{BuckmasterColomboVicol}) demonstrated non-uniqueness of distributional solutions to the Cauchy problem for the Navier-Stokes equations via the convex integration method for any $L^2$ initial datum. Their proof introduces the important tool of \emph{intermittency} and builds on the fundamental works for the Euler equations and the proof of the Onsager conjecture~\cite{IsettOnsager,OnsagerAdmissible} (see~\cite{convexintegrationconstructionsinturbulence,buckmaster2021non} for further connections to fluid turbulence). The solutions of~\cite{BuckmasterVicolAnnals} {have finite kinetic energy ($u \in L^\infty_t L^2_x$) but do not belong to the energy class~\eqref{eq:minimumregularity}}; more generally, all known convex integration schemes are far from reaching the regularity $\nabla u \in L^2_{t,x}$ necessary to speak of Leray solutions.  

Non-uniqueness of Leray solutions was instead proved in~\cite{mariahypodissipativeonefifth,DeRosa19}
for the \emph{hypodissipative} Navier-Stokes equations, namely, with the Laplacian in \eqref{eq:ns} replaced by a fractional Laplacian $(-\Delta)^{\alpha}$ of order $\alpha < 1/3$.
In forthcoming work~\cite{DallasMariaFractional} by the first and third authors, we  demonstrate non-uniqueness of Leray solutions of the forced Navier-Stokes equations in two dimensions in the whole supercritical range $\alpha < 1$. As in the present work, our methods are based on self-similarity and instability.

Recently, convex integration was also used~\cite{cheskidov2020sharp,cheskidovl2critical} 
to demonstrate the sharpness of certain classical results, obtaining in particular non-uniqueness in $L^q_t L^\infty_x$, $q < 2$, in any dimension $d \geq 2$, and in $C_t L^p_x$, $p < 2$, in dimension $d=2$, respectively.

\emph{Vortex filaments}. Our vortex ring construction is partially inspired by, though rather different from, the works~\cite{GallaySverakUniqueness2019,BedrossianVortexFilaments2018} on vortex filaments. In those papers, existence and uniqueness were established for a class of Navier-Stokes solutions whose initial vorticity is a Dirac mass along a curve. The infinitesimal structure of these vortex filaments is the two-dimensional \emph{Lamb--Oseen vortex}, whose stability was demonstrated in~\cite{GallayWayne} (see also~\cite{gallaghergallay,gallaghergallaylions}). We mention also~\cite{gallaysverakremarks}, which, motivated by vortex rings, discusses analogies between the two-dimensional and axisymmetric-without-swirl Navier-Stokes equations from the perspective of well-posedness.



\emph{Open problems}. As mentioned above, the main task of this paper is to find an unstable (forced) Euler steady state. The steady state we choose is a perturbation of a particular (unforced) unstable vortex of Vishik whose construction is far from simple. Given the flexibility afforded by a force, it would be interesting to know how generic 
this instability is and to develop a general method for finding these unstable (forced) solutions. It is, moreover, an open problem to remove the force.

\section{Linear instability: from two dimensions to axisymmetry}
\label{sec:twodimtoaxisym}

\subsection{Spectral preliminaries}\label{sec:spectralprelim}

Let $H$ be a Hilbert space and $\L : D(\L) \subset H \to H$ be closed and densely defined. Then $\D(L)$ is a Hilbert space with the graph norm. The basic theory of these operators is reviewed in \cite[VIII: Unbounded Operators]{ReedSimonI}. A simple criterion for self-adjointness is that $\ker \L + i$ and $\ker \L - i$ are trivial. 

The \emph{resolvent set} $\rho(\L)$ consists of those $\lambda \in \mathbb{C}$ such that $\lambda - \L : D(\L) \to H$ is invertible. Its inverse $R(\lambda, \L) : H \to D(\L)$ is known as the \emph{resolvent}, and it is bounded by the open mapping theorem. The \emph{spectrum} $\sigma(\L)$ is the complement of the resolvent set.
In our convention, the \emph{essential spectrum} $\sigma_{\rm ess}(\L)$ is the set of spectral values for which $\lambda - \L$ is not Fredholm of index zero.\footnote{The essential spectrum is sometimes defined otherwise, as the set where $\lambda - \L$ is not semi-Fredholm~\cite{Katobook}, or as the set where $\lambda - \L$ is not Fredholm~\cite{EngelNagel}.} The essential spectrum is closed. Moreover, \cite[Theorem 5.28]{Katobook} (see also the discussion in \cite[Chapter 4, Section 6]{Katobook}) can be used to show that, in each connected component of $\C \setminus \sigma_{\rm ess}(\L)$,
(i) $\lambda - \L$ is not invertible everywhere, or (ii) $\lambda - \L$ is invertible except at isolated points.



An operator $\T : D(\L) \to H$ is \emph{relatively compact with respect to $\L$} if $\T$ is compact when $D(\L)$ is considered with the graph norm. A key property is that relatively compact perturbations do not disturb the essential spectrum: $\sigma_{\rm ess}(\L+\T) = \sigma_{\rm ess}(\L)$, see \cite[Theorem 5.26]{Katobook}.

 We frequently use the following reasoning: Consider an operator $\L$ as above whose essential spectrum is known to be contained in a half space $\{ \Re \lambda \leq \mu \}$. We perturb $\L$ by a relatively compact perturbation $\T$, so that $\sigma_{\rm ess}(\L + \T) = \sigma_{\rm ess}(\L)$. Furthermore, we typically know that the perturbed operator $\L + \T$ is invertible for $\Re \lambda \gg 1$. Hence, the spectrum in $\{ \Re \lambda > \mu \}$ consists of isolated points.
 

\subsection{Two-dimensional instability}

To begin, we recall the unstable vortex constructed in~\cite{Vishik2}, see also~\cite{OurLectureNotes}.

For integer $m \geq 2$, we define the following space of $m$-fold rotationally symmetric functions:
\begin{equation} \label{def:m-fold}
    L^2_m := \{ f \in L^2(\R^2) :  f (R_{\frac{2\pi}{m}} x) = f(x) \, \mbox{for a.e. }x\in \R^2 \},
\end{equation}
where $R_{\frac{2\pi}{m}}:\R^2 \to \R^2$ is counterclockwise rotation by $2\pi/m$.

Typically, the two-dimensional Biot-Savart law $\BS_{2d}[f] := \nabla^\perp \Delta^{-1} f = K_2 \ast f$, considered on $L^2(\R^2)$, is only well defined \emph{up to constants} and satisfies $L^2(\R^2) \to \dot H^1(\R^2;\R^2)$. An important point is that the Biot-Savart law is well defined on $L^2_m$ \emph{within the class of $m$-fold rotationally symmetric functions}, since the $m$-fold rotational symmetry sets the mean velocity on $B_R$ to zero for all $R > 0$. 



For $f\in L^2(\R^2)$, using polar coordinates $(x,y)=(\varrho \cos \varphi, \varrho \sin \varphi)$ and expanding in Fourier series, we can write
\begin{equation}
    f(x,y) = \sum_{k\in \mathbb{Z}} f_k(\varrho) e^{ik \varphi} \, ,
    \qquad
     f_k(\varrho) := \frac{1}{2\pi} \int_0^{2\pi} f(\varrho, \varphi) e^{-ik\varphi} d \varphi \, .
\end{equation}
The Plancherel identity gives
\begin{equation}
    \| f \|_{L^2(\R^2)}^2 = 2\pi \sum_{k\in \mathbb{Z}} \| f_k \|_{L^2(\R_+, \, \varrho \, d\varrho)}^2 \, .
\end{equation}
It is immediate to check that $f\in L^2_m$ if and only if $f_k = 0$ whenever $m$ does not divide $k$. It is convenient to introduce the following decomposition of $L^2_m$:
\begin{equation}\label{eq: L2 decomp}
    L^2_m = \oplus_{k \in \Z} U_k \, ,
    \qquad
    U_k := \{ f(x) = g(\varrho) e^{ikm\varphi} \, : \, g\in L^2(\R_+, \, \varrho \, d\varrho) \} \, ,
\end{equation}
where $U_k$ are closed, mutually orthogonal subspaces of $L^2_m$.

\medskip

Let us consider a smooth, divergence-free vector field $\bar u(x) = \zeta(|x|) x^{\perp}$ that decays at infinity, where $x^\perp=(-x_2, x_1)$. We set
\begin{equation}
    \bar \omega = \curl \bar u = \nabla^\perp \cdot \bar u \, ,
\end{equation}
which is a radial function.


We then define
\begin{equation}
    \label{eq:adef1}
    \A : D(\A) \subset L^2_m \to L^2_m, \quad D(\A) := \{ \omega \in L^2_m : \bar{u} \cdot \nabla \omega \in L^2_m \},
\end{equation}
\begin{equation}
    \label{eq:adef2}
- \A \omega = \bar{u} \cdot \nabla \omega + u \cdot \nabla \bar{\omega},
\end{equation}
where $u = \BS_{2d}[\omega]$. 
It is verified in~\cite{Vishik1,Vishik2,OurLectureNotes} that $\K\omega := u \cdot \nabla \bar{\omega}$ is a compact operator $\K : L^2_m \to L^2_m$. The operator $\A$ is closed and densely defined. By regarding it as a compact perturbation of the skew-adjoint operator $\omega \mapsto \bar{u} \cdot \nabla \omega$, we have that the essential spectrum of $\A$ belongs to the imaginary axis, and the remainder of the spectrum is discrete, consisting of countably many eigenvalues of finite algebraic multiplicity.

We now state a version of Vishik's theorem in~\cite{Vishik2}, see also \cite[Theorem 3.0.1]{OurLectureNotes}.

\begin{theorem}
    \label{thm:vishiksunstableprofile}
There exists $m \geq 2$ and a smooth, radially symmetric vorticity profile $\bar \omega = \bar \omega(\varrho)$ satisfying
\begin{equation}
    |\bar{\omega}| + \varrho |\p_{\varrho} \bar{\omega}| \les \la \varrho \ra^{-2}
\end{equation}
such that $\A \: D(\A) \subset L^2_m \to L^2_m$ has an unstable eigenvalue $\lambda$ (that is, $\Re{\lambda} > 0$).
\end{theorem}
 The profile can be chosen to satisfy many additional properties, which we do not require here. {Recall that $\la \varrho \ra = (1 + \varrho^2)^{1/2}$ is the Japanese bracket notation.} {The decay rate $\la \varrho \ra^{-2}$ is not essential; any power-law decay will suffice.} One may multiply the vortex by a prefactor $\beta \gg 1$ to increase $\Re \lambda$.

We now demonstrate that appropriate truncations of unstable vortices remain unstable.
Let $\phi \in C^\infty_0(B_1)$ be radially symmetric with $\phi \equiv 1$ on $B_{1/2}$. For $R>0$, consider $\phi_R(x) := \phi(x/R)$ and the truncation
\begin{equation}
    \bar{u}_R = \phi_R \bar{u}, \quad \bar{\omega}_R = \curl \bar{u}_R = \nabla^\perp \cdot \bar{u}_R.
\end{equation}
We consider also operators $\A_R$ defined analogously to $\A$ in~\eqref{eq:adef1}-\eqref{eq:adef2} with $\bar{u}_R$ replacing $\bar{u}$ and $\bar{\omega}_R$ replacing $\bar{\omega}$.

\begin{proposition}[Truncated unstable vortex]
\label{pro:truncatedunstablevortex}
Let $\bar{\omega}$ be an unstable vortex as in Theorem~\ref{thm:vishiksunstableprofile}. 
Let $\lambda_\infty$ be an unstable eigenvalue of $\A$. For all $\varepsilon \in (0, \Re \lambda_\infty)$, there exists $R_0 > 0$, depending on $\varepsilon$, $\lambda_\infty$, and $\bar \omega$, such that for all $R \geq R_0$, $\A_R$ has an unstable eigenvalue $\lambda_R$ with $|\lambda_R - \lambda_\infty| < \varepsilon$.
\end{proposition}

\begin{proof}
Let us begin by observing (separation of variables) that the subspaces $U_k$ defined in \eqref{eq: L2 decomp} are invariant under the operators $\A$ and $\A_R$. Moreover, the restrictions
\begin{equation}
    \A^{(k)} = \A|_{U_k} \, , \qquad
    \A^{(k)}_R = \A_R|_{U_k} \, ,
\end{equation}
are continuous. 
Indeed, for $\omega(x) = g(\varrho) e^{i mk\varphi}$ we have
\begin{equation}
    \A^{(k)} \omega = \left( - \zeta imk g + \frac{1}{\varrho} imk f \bar{\omega}' \right) e^{imk\varphi},
\end{equation}
\begin{equation}
    \A^{(k)}_R \omega = \left( - \phi_R \zeta imk g + \frac{1}{\varrho} imk f \bar{\omega}_R' \right) e^{imk\varphi},
\end{equation}
where 
\begin{equation}
\left( \p_{\varrho\varrho}  + \frac{1}{\varrho} \p_{\varrho}  - \frac{(mk)^2}{\varrho^2} \right) f = g \, .
\end{equation}
By employing \cite[Lemma 4.12]{OurLectureNotes}, one can see that
\begin{equation}\label{zzz1}
    \| f \|_{L^\infty(\varrho \, d\varrho)} \les_{mk} \| g \|_{L^2(\varrho \, d\varrho)} \les \| \omega \|_{L^2_m} \, .
\end{equation}


We now demonstrate that
\begin{equation}
    \label{eq:akconverges}
    \A^{(k)}_R \overset{R \to +\infty}{\longrightarrow} \A^{(k)}
\end{equation}
in operator norm. 
Let $n := km \neq 0$, since the operators $\A^{(0)}$ and $\A^{(0)}_R$ are zero. For $\omega(x) = g(\varrho) e^{in\varphi}$ and $f(\varrho)$ as above, the difference between $\A^{(k)}_R$ and $\A^{(k)}$ is
\begin{equation}
    \left( \A^{(k)}_R - \A^{(k)} \right) \omega = \left( - \zeta (\phi_R - 1) in g + \frac{1}{\varrho}  in f (\bar{\omega}_R' - \bar{\omega}') \right) e^{in\varphi} \, .
\end{equation}
Since
\begin{equation}
    \| \bar{u}^\varphi (\phi_R - 1 ) \|_{L^\infty(\varrho \, d\varrho)} + \| \frac{1}{\varrho} (\bar{\omega}_R' - \bar{\omega}') \|_{L^2(\varrho \, d\varrho)} \overset{R \to +\infty}{\longrightarrow} 0 \, 
\end{equation}
due to the cutoff $\phi_R$ and the decay properties of $\bar{u}$ and $\bar{\omega}$, the claimed conclusion~\eqref{eq:akconverges} follows from \eqref{zzz1} and H{\"o}lder's inequality. 

The unstable eigenvalue $\lambda_\infty$ can be associated to a non-trivial eigenfunction $\eta$ belonging to $U_{k_0}$ for some $k_0 \in \Z$. Since $\A^{(0)}$ is trivial, we must have that $k_0 \neq 0$. 

Let $K$ be a compact subset of $\rho(\A) \cap \{ \Re \lambda > 0 \}$. For all $\lambda \in K$, we have
\begin{equation}
\begin{aligned}
            \lambda - \A^{(k)}_R &= \lambda - \A^{(k)} + (\A^{(k)} - \A^{(k)}_R) \\
            &= (\lambda - \A^{(k)}) [I + R(\lambda,\A^{(k)}) (\A^{(k)} - \A^{(k)}_R) ].
\end{aligned}
\end{equation}
Due to~\eqref{eq:akconverges}, there exists $R_0 > 0$ (which depends on $K$) such that, for all $R \geq R_0$, $\lambda - \A^{(k)}_R$ is invertible, its inverse is
\begin{equation}
    R(\lambda,\A^{(k)}_R) = [I + R(\lambda,\A^{(k)}) (\A^{(k)} - \A^{(k)}_R) ]^{-1} R(\lambda,\A^{(k)}),
\end{equation}
and
\begin{equation}
    R(\lambda, \A^{(k)}_R) \overset{R \to +\infty}{\longrightarrow} R(\lambda,\A^{(k)})
\end{equation}
in operator norm, uniformly on $K$.

Let $\vec{c} \subset \rho(\A^{(k_0)}) \cap \{ \Re \lambda > 0 \}$ be a counterclockwise, smooth, simple, closed curve containing $\lambda_\infty$ in its interior and no other element of $\sigma(\A^{(k_0)})$. Let $K = \vec{c}$. We define the spectral projection $\Pr_R^{(k_0)}$ onto $\lambda_\infty$ according to
\begin{equation}
    \Pr_R^{(k_0)}  = \frac{1}{2\pi i} \int_{\vec{c}}  R(\lambda,\A^{(k)}_R) \, d\lambda,
\end{equation}
and we define $\Pr^{(k_0)}$ similarly. The projection $\Pr^{(k_0)}$ is non-trivial. Therefore, since $\Pr^{(k_0)}_R \to \Pr^{(k_0)}$ in operator norm as $R \to +\infty$, we must have that for all $R$ sufficiently large, $\Pr^{(k_0)}_R$ is also non-trivial. Therefore, $\A^{(k_0)}_R$ must have an unstable eigenvalue, and so must $\A_R$. This completes the proof.
\end{proof}



For $\varepsilon = \Re \lambda_\infty / 2$, we choose $\bar{R} \geq R_0$ and define the truncated vortex $\tilde{u} = \bar{u}_{\bar{R}}$ and $\tilde{\omega} = \curl \tilde{u}$. Let us introduce the weighted space $L^2_\gamma = L^2(\R^2; \gamma \, dr \, dz)$, without imposing $m$-fold symmetry, where $\gamma \geq 0$ is a smooth weight satisfying
\begin{equation}\label{eq: gamma}
    \gamma \equiv 1 \text{ on } B_{\bar{R}}, \quad \la r, z \ra^{100} \les |\gamma| \les \la r, z \ra^{100} \text{ on } \R^2 \, ,
\end{equation}
where $\la \gamma,z \ra^2 := \gamma^2 + z^2 + 1$.
We now consider the operator $\L_\infty \: D(\L_\infty)\subset  L^2_\gamma \to L^2_\gamma$ defined by
\begin{equation}
    - \L_\infty : = \tilde u \cdot \nabla \omega + u\cdot \nabla \tilde \omega \, ,
\end{equation}
where $u = \BS_{2d}[\omega]$. Notice that $\BS_{2d} : L^2_\gamma \to L^{2,\infty} \cap \dot H^1$ is well defined and continuous since $L^2_\gamma$ embeds into $L^1$. The domain of $\L_\infty$ is
\begin{equation}
    D(\L_\infty) := \{ \omega \in L^2_\gamma : \tilde{u} \cdot \nabla \omega \in L^2(B_{\bar{R}}) \}\, .
\end{equation}

Using that $\supp \tilde{u}, \supp{\tilde{\omega}} \subset B_{\bar{R}}$, it is not difficult to verify that the operator $\omega \mapsto u \cdot \nabla \tilde{\omega}$ is compact. Hence, for the same reasons as before, the essential spectrum of $\L_\infty$ is on the imaginary axis, and the remainder of the spectrum is discrete, consisting of countably many eigenvalues of finite algebraic multiplicity. 

\begin{corollary}[Instability in weighted space]
\label{cor:instabilityinweighted}
The operator $\L_\infty$ has an unstable eigenvalue. 
\end{corollary}

\begin{proof}
Let $\eta$ be a non-trivial unstable eigenfunction of $\A_{\bar{R}}$. Since $\supp (\A_{\bar{R}} \eta) \subset B_{\bar{R}}$ and $\lambda \eta + \A_{\bar{R}} \eta = 0$, it follows that $\eta$ is supported in $B_{\bar R}$. In particular $\eta \in L_\gamma$ and $\lambda \eta + \L_\infty \eta = 0$.
\end{proof}


\subsection{Axisymmetric instability without swirl}

We now pass to the axisymmetric situation. We artificially place the two-dimensional vortex $\tilde{u}(x,y) = \tilde{u}^x e_x + \tilde{u}^y e_y$ into the axisymmetric-without-swirl coordinates $(r',z) \in \R_+ \times \R$. More often, we work with the variable $r = r' - \ell \in \R_{> -\ell}$, where $\ell \geq 2\bar{R}$. We define 
\begin{equation}
\tilde{u}(r,z) = \tilde{u}^r e_r + \tilde{u}^z e_z,
\end{equation}
where, as functions on $\R^2$, $\tilde{u}^r(r,z) = \tilde{u}^x(r,z)$ and $\tilde{u}^z(r,z) = \tilde{u}^y(r,z)$.
So far, $\tilde{u}$ does not correspond to a \emph{divergence-free} axisymmetric-no-swirl vector field, since the physical divergence of $\tilde{u}$ is
\begin{equation}
    \div_\ell \tilde{u} := \left( \p_r + \frac{1}{r+\ell} \right) \tilde{u}^r + \p_z \tilde{u}^z = (r+\ell)^{-1} (\p_r,\p_z) \cdot \left[ (r+\ell) \tilde{u} \right].
\end{equation}
Since $(\p_r,\p_z) \cdot \tilde{u} = 0$, there is an error of $(r+\ell)^{-1} \tilde{u}^r$, which we correct below:

\begin{lemma}[Correcting the divergence]\label{lemma:div}
For each $\ell \geq 2\bar{R}$, there exists $v_\ell \in C^\infty_0(B_{\bar{R}};\R^2)$ satisfying
\begin{equation}
    \label{eq:equationsatisfiedbyvell}
    \div_{\ell} \left( \tilde{u} + v_\ell \right) = 0
\end{equation}
and
\begin{equation}
    v_\ell \overset{\ell \to +\infty}{\longrightarrow} 0 \text{ in } C^k(B_{\bar{R}}) \text{ for all } k \geq 0.
\end{equation}
\end{lemma}
\begin{proof}
We multiply~\eqref{eq:equationsatisfiedbyvell} through by $r+\ell$ to find that the correction $v_\ell$ must solve
\begin{equation}\label{eqn:divvell}
    (\p_r,\p_z) \cdot [ (r+\ell) v_\ell ] = - \tilde{u}^r.
\end{equation}
The right-hand side belongs to $C^\infty_0(B_{\bar{R}})$ and has zero mean due to the radial symmetry of $\tilde{u}$: $\int_{\R^2} \tilde{u}(x,y) \, dx \, dy = 0$. Bogovskii's operator~\cite{bogovskii} (see also Galdi's book~\cite[Section III.3]{galdi}) gives us a solution $(r+\ell) v_\ell \in C^\infty_0(B_{\bar{R}};\R^2)$ which is independent of~$\ell$. The function $v_\ell$ and its quantitative bounds are recovered after dividing by $r+\ell$. \end{proof}

We define the corrected background as
\begin{equation}
    \tilde{u}_\ell = \tilde{u} + v_\ell, \quad \tilde{\omega}_\ell = \curl_{\ell} \tilde{u} := -\p_z \tilde{u}^r + \p_r \tilde{u}^z
    \, .
\end{equation}
Here, we identify the $3d$ curl of $\tilde u\in C^\infty(\R^3; \R^3)$ with a scalar function $\curl_{\ell} \tilde{u}$. This is possible because, as $\tilde u$ is axisymmetric without swirl, it holds
\begin{equation}
     \curl_{3d} \tilde u = - \curl_\ell \tilde u \, e_\theta \, . 
\end{equation}
Notice the sign convention for the swirl component we introduced above~\eqref{eq:eulerintro}.

Consider the linearized axisymmetric-no-swirl Euler operator around the background $\tilde{u}_\ell$:
\begin{equation}
    -\L_\ell \omega := \tilde{u}_\ell \cdot \nabla \omega + u \cdot \nabla \tilde{\omega}_\ell - \frac{\tilde{u}_\ell^r}{r+\ell} \omega - \frac{u^r}{r+\ell} \tilde{\omega}_\ell, 
    \quad u = \BS_\ell[\omega],
\end{equation}
where here and in this subsection $\nabla = (\p_r,\p_z)$ and $u = \BS_\ell[\omega]$ is defined by
\begin{equation}
    \p_r^2 \psi + \frac{1}{r+\ell} \p_r \psi - \frac{1}{(r+\ell)^2} \psi + \p_z^2 \psi = \omega \text{ in } \R^2_{> -\ell}, \quad \p_n \psi|_{r = -\ell} = 0
\end{equation}
\begin{equation}
    u = -\p_z \psi \, e_r + \left( \p_r + \frac{1}{r+\ell} \right) \psi \, e_z.
\end{equation}
We emphasize here that $\omega$ is regarded as minus the swirl component of the vorticity and $\supp \tilde{u}_\ell,\, \supp \tilde{\omega}_\ell \subset B_{\bar{R}}$ and

\subsubsection{Functional analytic set-up}
We consider the space
\begin{equation}
    L^2_{\gamma,\ell} := L^2(\R^2_{r > -\ell} ; \gamma \, dr \, dz) \, ,
\end{equation}
where $\gamma$ is defined in \eqref{eq: gamma}.
We define the operator $\L_\ell \: D(\L_\ell) \subset L^2_{\gamma,\ell} \to L^2_{\gamma,\ell}$ with domain
\begin{equation}
    D(\L_\ell) := \{ \omega \in L^2_{\gamma,\ell} : \tilde{u}_\ell \cdot \nabla \omega \in L^2(B_{\bar{R}}) \}
\end{equation}
according to
\begin{equation}
     \L_\ell = \M_\ell + \K_\ell + \S_\ell \, ,
\end{equation}
where
\begin{equation}
    -\M_\ell \omega := \tilde{u}_\ell \cdot \nabla \omega + \frac{1}{2} (\div_{2d} v_\ell) \omega \, ,
\end{equation}
\begin{equation}
    -\K_\ell \omega := \BS_\ell[\omega] \cdot \nabla \tilde{\omega}_\ell \, ,
\end{equation}
 \begin{equation}
        -\S_\ell \omega := - (r+\ell)^{-1} \BS_{\ell}[\omega]^r \tilde{\omega}_\ell - (r+\ell)^{-1} \tilde{u}^r_\ell \omega - \frac{1}{2} (\div_{2d} v_\ell) \omega \, .
\end{equation}
It is not hard to see (Lemma~ \ref{lem:basicproperties} below) the mapping properties
\begin{equation}
    \M_\ell : D(\L_\ell) \subset L^2_{\gamma,\ell} \to L^2(B_{\bar{R}}) \subset L^2_{\gamma,\ell} \, ,
\end{equation}
and 
\begin{equation}
    \K_\ell\, , \S_\ell : L^2_{\gamma,\ell} \to L^2(B_{\bar{R}}) \subset L^2_{\gamma,\ell} \, .
    \end{equation}
The naming convention is as follows: $\M_\ell$ is the `main term', which is also skew-adjoint with the above choice of domain; $\K_\ell$ is `compact'; and $\S_\ell$ is `small'. 
To show that $\M_\ell$ is~skew-adjoint we use that $\gamma = 1$ on the support of $\tilde u_\ell$ and $\tilde \omega_\ell$.

\begin{lemma}[Basic properties]
\label{lem:basicproperties}
The operators $\K_\ell, \S_\ell : L^2_{\gamma, \ell}\to L^2_{\gamma,\ell}$ are bounded.
Moreover, $\K_\ell$ is compact, and 
\begin{equation}
\label{eqn:term3}
\| \S_\ell \|_{L^2_{\gamma,\ell} \to L^2_\gamma}  \to 0 \text{ as } \ell \to +\infty  \, .
\end{equation}
\end{lemma}

Since $\bar{R}$ is fixed, it will be convenient to omit the dependence of the implied constants on $\bar{R}$ in the remainder of the paper.

\begin{proof}
Let us begin by proving \eqref{eqn:term3}.
Let $\omega$ be a function supported on ${\{r>-\ell\}}$.
By \eqref{eqn:divvell}, we have that 
$\div_{2d} v_\ell=- (r+\ell)^{-1} v_\ell^r  - (r+\ell)^{-1}  \tilde{u}^r$ and, by the boundedness of $ \tilde{u}^r$ and since $v_\ell^r$ goes to $0$ in $L^\infty$, we deduce that
$\div_{2d} v_\ell $ goes to $0$ in $L^\infty$.
Using also that $\tilde{u}^r_\ell$ is supported in $B_{\bar R}$, we estimate
\begin{equation}\label{eqn:t1}
\begin{split}
\left\|  (r+\ell)^{-1} \tilde{u}^r_\ell \omega +
 \frac{1}{2} (\div_{2d} v_\ell) \omega \right\|_{ L^2_{\gamma,\ell}} 
 & =
 \left\|  (r+\ell)^{-1} \tilde{u}^r_\ell \omega +
 \frac{1}{2} (\div_{2d} v_\ell) \omega \right\|_{ L^2(B_{\bar{R}})} 
 \\ & \le 
 \underbrace{ \left( \frac 1 {\ell-\bar R} \|\tilde{u}^r_\ell\|_{L^\infty} +  \|\div_{2d} v_\ell\|_{L^\infty} \right) }_{\to 0 \text{ as } \ell \to +\infty} \| {\omega} \|_{ L^2_{\gamma,\ell}} \, .
\end{split}
\end{equation}

The standard estimate for the Biot-Savart kernel in $\R^3$, read on axisymmetric functions, gives us that
\begin{equation}
\label{eqn:BSl-est}
\| \BS_\ell[\omega]\|_{L^{6}((r+\ell)\,dr\,dz)} 
\les \| \omega \|_{L^2((r+\ell)\,dr\,dz)} 
\les \ell^{1/2}  \| \omega \|_{L^2(\gamma\,dr\,dz)} \, .
\end{equation}
Hence, we estimate
\begin{equation}\label{eqn:t2}
\begin{split}
\| (r+\ell)^{-1} \BS_{\ell}[\omega]^r \tilde{\omega}_\ell \|_{ L^2_{\gamma, \ell}} 
& =
\| (r+\ell)^{-1} \BS_{\ell}[\omega]^r \tilde{\omega}_\ell \|_{ L^2(B_{\bar{R}})} 
\\& \les 
\frac 1 {\ell-\bar R} \|  \BS_{\ell}[\omega]\|_{ L^6(B_{\bar{R}})}  \| \tilde{\omega}_\ell \|_{ L^3(B_{\bar{R}})} 
\\& \les
\ell^{-5/6} \| \BS_{\ell}[\omega]\|_{ L^6((r+\ell)\,dr\,dz)}  \| \tilde{\omega}_\ell \|_{ L^3(B_{\bar{R}})} 
\\& \les
\ell^{-1/3} \| \omega \|_{L^2_{\gamma,\ell}} \| \tilde{\omega}_\ell \|_{ L^3(B_{\bar{R}})}  \, .
\end{split}
\end{equation}
Recalling that $\tilde{\omega}_\ell $ is bounded independently on $\ell$, putting together \eqref{eqn:t1}, \eqref{eqn:t2} and \eqref{eqn:BSl-est}, we obtain \eqref{eqn:term3}.

To prove that $\K_\ell$ is compact, it is enough to check that 
\begin{equation}
    \| \BS_\ell[\omega] \|_{L^6(B_{\bar R})} + \| \nabla \BS_\ell[\omega] \|_{L^2(B_{\bar R})}
    \les_{\ell} 1 \, .
\end{equation}
The $\dot H^1$ estimate follows from the fact that $\omega \in L^2$, while the $L^6$ estimate is a consequence of \eqref{eqn:BSl-est}.
\end{proof}

With these properties in mind, we can say the following: The essential spectrum of $\L_\ell$ is contained in a half-plane $\{ \Re \lambda < o_{\ell \to +\infty}(1) \}$, outside of which the spectrum is discrete, consisting of countably many eigenvalues of finite algebraic multiplicity. We always consider $\ell$ large enough so that the $o(1)$ term above is $\leq 1/2$.

\subsubsection{Main result} 
We are now ready to state our main instability result for the operator $\L_\ell$. The proof is given at the end of this section.

\begin{proposition}[Axisymmetric instability]
    \label{pro:axisymmetricinstab}
Let $\lambda_\infty$ be an unstable eigenvalue of $\L_\infty$. 
For all $\varepsilon \in (0,\Re \lambda_\infty)$, there exists $\ell_0 \geq 2\bar{R}$, depending on $\tilde{u}$, $\varepsilon$, and $\lambda_\infty$, such that for all $\ell \geq \ell_0$ , the operator $\L_\ell$ has an unstable eigenvalue $\lambda_\ell$ with $|\lambda_\ell - \lambda_\infty| < \varepsilon$.
\end{proposition}

It will be convenient to consider all operators on the single function space $L^2_\gamma$ rather than $L^2_{\gamma,\ell}$. Define the projection operator $P_\ell : L^2_{\gamma} \to L^2_{\gamma,\ell}$ by restriction:
\begin{equation}
    P_\ell \omega = \omega|_{r > -\ell} \, .
\end{equation} 
We consider the operator 
\begin{equation}
    \L_\ell P_\ell = \M_\ell P_\ell + \K_\ell P_\ell +  \S_\ell P_\ell : L^2_{\gamma} \to L^2_\gamma \, ,
\end{equation}
and observe that
for $\ell \ge 2 \bar R$ and $\lambda \in \C \setminus \{ 0 \}$,
\begin{equation}\label{lemma:projection}
\lambda \in \sigma(\L_\ell) \qquad    \mbox{if and only if} \qquad \lambda\in \sigma(\L_\ell P_\ell).
\end{equation}
Indeed, if we make the identification
\begin{equation}
    L^2_\gamma = L^2( \{ r < -\ell \} \times \R, \, \gamma \, dr \,dz) \oplus L^2_{\gamma,\ell} \, ,
    \quad
    \omega = ( (I-P_\ell)\omega, P_\ell \omega ) = (f,g) \, ,
\end{equation}
the operator $\lambda + \L_\ell P_\ell$ acts as $(f,g) \mapsto (\lambda f, \lambda g + \L_\ell g)$. Hence, $\lambda + \L_\ell P_\ell$ is invertible if and only if $\lambda + \L_\ell$ is invertible. This implies~\eqref{lemma:projection}.


We introduce the operators
\begin{equation}
    \K_\infty : L^2_\gamma \to L^2_\gamma \, , 
    \qquad
    -\K_\infty \omega := \BS_{2d}[\omega]\cdot \nabla \tilde \omega \, ,
\end{equation}
\begin{equation}
    \M_\infty : D(\M_\infty) \subset L^2_\gamma \to L^2_\gamma \, ,
    \qquad
    -\M_\infty \omega := \tilde u \cdot \nabla \omega \, ,
\end{equation}
where $D(\M_\infty) = \{\omega \in L^2_\gamma : \tilde u \cdot \nabla \omega \in L^2(B_{\bar{R}})\}$.


\begin{lemma}
\label{lem:auxlem}
Let $K$ be a compact subset of $\rho(\M_\infty + \K_\infty) \cap \{ \Re \lambda > 0 \}$. There exists $\ell_0 \geq 2\bar{R}$ such that, for all $\ell \geq \ell_0$, the operator $\lambda - \M_\ell P_\ell - \K_\infty$, acting on $L^2_\gamma$, is invertible for all $\lambda \in K$, and
\begin{equation}
    \label{eq:desiredconvergenceproperty}
    R(\lambda, \M_\ell P_\ell  + \K_\infty) f \overset{\ell \to +\infty}{\longrightarrow} R(\lambda, \M_\infty  + \K_\infty) f \quad \text{ in } L^2_\gamma \, ,
\end{equation}
for all $f \in L^2_\gamma$, uniformly on $K$.
\end{lemma}

\begin{proof}
First, we show that
\begin{equation}
    \label{eq:modeofconvergenceone}
    R(\lambda, \M_\ell P_\ell) f \overset{\ell \to +\infty}{\longrightarrow} R(\lambda, \M_\infty) f \quad \text{ for all } f \in L^2_\gamma
\end{equation}
locally uniformly in $\{ \Re \lambda > 0 \}$.
To do so, we employ the identity
\begin{equation}
    [R(\lambda,  \M_\ell P_\ell) - R(\lambda,  \M_\infty)]f(x)
    = \int_0^\infty e^{-\lambda t} (f(X^\ell_t(x)) - f(X_t(x))) d t \, ,
\end{equation}
where $X^\ell$ and $X$ are the flow maps associated to, respectively, $\tilde u_\ell$ and $\tilde u$. Since $X_t^\ell(x) = X_t(x) = x$ for any $x\in \R^2\setminus B_{\bar R}$, it holds
\begin{equation}
    \| [R(\lambda,  \M_\ell P_\ell) - R(\lambda,  \M_\ell P_\ell)] f \|_{L^2_\gamma}
      \le \int_0^\infty e^{- \Re \lambda t} \| f(X_t^\ell) - f(X_t)\|_{L^2(B_{\bar R})} d t \, .
\end{equation}
The claimed conclusion \eqref{eq:modeofconvergenceone} follows from a simple approximation argument and the fact that
$X_t^\ell \overset{\ell \to +\infty}{\longrightarrow}X_t$ in $C^1$ as a consequence of Lemma~\ref{lemma:div}.


We now show that
\begin{equation}\label{zzz2}
    I - R(\lambda, \M_\ell P_\ell) \K_\infty  \overset{\ell \to +\infty}{\longrightarrow} I - R(\lambda, \M_\infty) \K_\infty
\end{equation}
in operator norm,
locally uniformly in $\{ \Re \lambda > 0 \}$.

To this aim we use that $\K_\infty$ is a compact operator on the separable Hilbert space $L^2_\gamma$; hence, it can be approximated in operator norm by a sequence $(F_k)$ of finite-rank operators on $L^2_\gamma$.
Due to~\eqref{eq:modeofconvergenceone}, we have that
\begin{equation}
    R(\lambda, \M_\ell P_\ell) F_k \overset{\ell \to +\infty}{\longrightarrow} R(\lambda, \M_\infty) F_k,
\end{equation}
in operator norm,
locally uniformly in $\{ \Re \lambda > 0 \}$,
for each $k \in \N$. Therefore, \eqref{zzz2} follows from a simple diagonal argument.

From \eqref{zzz2} it follows that there exists $\ell_0 \geq 2\bar{R}$ such that, for all $\ell \geq \ell_0$, we have that $\lambda - \M_\ell P_\ell - \K_\infty$ is invertible, and
\begin{equation}\label{zzz3}
    R(\lambda, \M_\ell P_\ell + \K_\infty) = [ I - R(\lambda, \M_\ell P_\ell) \K_\infty]^{-1} R(\lambda, \M_\ell P_\ell) \, .
\end{equation}
Indeed, the identity
\begin{equation}
    \lambda - \M_\ell P_\ell - \K_\infty = (\lambda - \M_\ell P_\ell) [I - R(\lambda,  \M_\ell P_\ell) \K_\infty] \, ,
\end{equation}
tells us that $\lambda - \M_\ell P_\ell - \K_\infty $ is invertible if and only if $I - R(\lambda,  \M_\ell P_\ell) \K_\infty$ is invertible. When $\lambda \in K$, the operator on the right-hand side of \eqref{zzz2} is invertible, and its inverse depends continuously on $\lambda$; therefore, when $\ell$ is sufficently large, we have that $I - R(\lambda,  \M_\ell P_\ell) \K_\infty$ is invertible,

To conclude the proof of \eqref{eq:modeofconvergenceone} it remains only to show that 
\begin{equation}
    [ I - R(\lambda, \M_\ell P_\ell) \K_\infty]^{-1}
    \overset{\ell \to +\infty}{\longrightarrow}
    [ I - R(\lambda, \M_\infty) \K_\infty]^{-1} \, ,
\end{equation}
which again is a consequence of \eqref{zzz2} and the invertibility of $I - R(\lambda, \M_\infty) \K_\infty$.




\end{proof}

\begin{proof}[Proof of Proposition~\ref{pro:axisymmetricinstab}]

Thanks to \eqref{lemma:projection} and the fact that $\sigma(\L_\ell)$  in  $\{ \Re \lambda > o_{\ell \to +\infty}(1)\}$ consists of eigenvalues with finite multiplicity, it is enough to show that $\sigma(\L_\ell P_\ell)\cap \{ \Re \lambda > o_{\ell \to +\infty}(1)\}$ is non-empty.

For any $\lambda\in \rho( \M_\infty + \K_\infty) \cap \{ \Re \lambda > 0 \}$ we decompose the operator $\lambda - \L_\ell P_\ell$ as
\begin{equation}
\begin{aligned}
        &\lambda - \M_\ell P_\ell - \K_\ell P_\ell - \S_\ell P_\ell \\
        &\quad = \lambda - \M_\ell P_\ell - \K_\infty - \underbrace{\K_\infty(P_\ell - I)}_{=: \I_\ell}  - \underbrace{(\K_\ell P_\ell - \K_\infty)P_\ell}_{=: \II_\ell} - \S_\ell P_\ell \\
        &\quad = (\lambda - \M_\ell P_\ell - \K_\infty) [I - R(\lambda, \M_\ell P_\ell  + \K_\infty) (\I_\ell + \II_\ell + \S_\ell P_\ell)]
\end{aligned}
\end{equation}
on $D(\L_\ell)$ when $\ell \geq \ell_0$, see Lemma~\ref{lem:auxlem}. We argue below that $\I_\ell$ and $\II_\ell$ converge to zero as $\ell \to +\infty$ (it is already known that $\S_\ell$ does so, see Lemma~\ref{lem:basicproperties}). The consequence is that $I - R(\lambda, \M_\ell P_\ell + \K_\infty) (\I_\ell + \II_\ell + \S_\ell P_\ell)$, hence, $\lambda - \L_\ell P_\ell$, will be invertible when $\ell$ is large enough.

\textbf{1. Term $\I_\ell$}. Let us show that
\begin{equation}
\label{eqn:term1}
\|\K_\infty(P_\ell - I) \|_{L^2_\gamma \to L^2_\gamma}  \to 0 \qquad \mbox{as } \ell \to +\infty \, . 
\end{equation}
For any $\omega \in L^2_\gamma$, it holds
\begin{align*}
    \|\K_\infty(P_\ell - I) \omega \|_{ L^2_\gamma} 
    & =
    \|\K_\infty(P_\ell - I) \omega \|_{ L^2(B_{\bar{R}})} 
    \\ &\leq 
    \|\BS_{2d}[\omega 1_{\{r<-\ell\}}] \|_{ L^2(B_{\bar{R}})}  \| \nabla \tilde \omega _\ell\|_{ L^\infty} \, .
\end{align*}
The second factor is bounded since $\tilde \omega$ is smooth with bounded support and $\curl v_\ell$ and its derivatives are bounded by Lemma~\ref{lemma:div}. As regards the first term, since $\BS_{2d}$ sends $L^q(\R^2)$ to $L^{q^*}(\R^2)$ for any $q\in (1,2)$, by H{\"o}lder's inequality, we have for any fixed $\eps \in (0,1)$ that
\begin{equation}
\begin{split}
 \|\BS_{2d}[\omega 1_{\{r<-\ell\}}] \|_{ L^2(B_{\bar{R}})}
&\les_{\eps} \| \omega 1_{\{r<-\ell\}}\|_{L^{1+\eps}}
\\& \les_{\eps} \| \omega \|_{L^2(\gamma \, dr\, dz)} \| \la r, z \ra^{-50(1+\eps)} 1_{\{r<-\ell\}}\|_{L^{ \frac{2 }{1-\eps}}}^{\frac{1}{1+\eps}} \, .
\end{split}
\end{equation}
Notice that the second factor on the right-hand side is $o_{\ell \to +\infty}(1)$.


\textbf{2. Term $\II_\ell$}.
We show that
\begin{equation}\label{eqn:term2}
    \| (\K_\ell P_\ell - \K_\infty) P_\ell \|_{L^2_\gamma \to L^2_\gamma} \to 0 \quad \text{ as } \ell \to +\infty.
\end{equation}
Let $\omega$ be a function supported on ${\{r>-\ell\}}$.
It is enough to prove that
\begin{equation}
    \| \BS_{2d}[\omega] - \BS_\ell[\omega]\|_{L^2(B_{\bar{R}})}
    \le C(\ell) \| \omega \|_{L^2_\gamma} \, ,
\end{equation}
where $C(\ell) = o_{\ell \to +\infty}(1)$.

Recall that
\begin{equation}
    \BS_{2d}[\omega] = - \partial_z \psi \,  e_r + \partial_r \psi \, e_z \, ,
\end{equation}
\begin{equation}
    \BS_\ell[\omega] =  -\p_z \psi_\ell \, e_r +  \p_r \psi_\ell \, e_z + \frac{1}{r+\ell} \psi_\ell \, e_z \, ,
\end{equation}
where $\psi$, $\psi_\ell$ are the stream functions associated to $\omega$; they satisfy
\begin{equation}
    \Delta_\ell \psi_\ell = \left(\p_r^2 + \frac{1}{r+\ell} \p_r - \frac{1}{(r+\ell)^2} + \p_z^2 \right) \psi_\ell = \omega \, ,
    \qquad
    \p_n \psi_\ell|_{r = -\ell} = 0 \, ,
\end{equation}
\begin{equation}
    \Delta \psi = (\p_r^2  + \p_z^2) \psi = \omega \, .
\end{equation}
In particular, we have
\begin{equation}
    \| \BS_{2d}[\omega] - \BS_\ell[\omega]\|_{L^2(B_{\bar{R}})}
    \le 
    \| (r+\ell)^{-1} \psi_\ell \|_{L^2(B_{\bar R})} + \| \nabla (\psi_\ell - \psi) \|_{L^2(B_{\bar R})} =: I + II \, .
\end{equation}
We now estimate $I$ and $II$ separately.

Let us begin by studying $I$.
We recall the standard Young convolution inequality applied to the Biot-Savart kernel in $\R^3$ gives 
\begin{equation}\label{eqn:CZ}
    \begin{split}
    \| \psi_\ell \|_{L^{q}( ({r+\ell}) \, dr \, dz )} 
    + &
    \|(\p_r + (r+\ell)^{-1}, \p_z) \psi_\ell \|_{L^{p^*}((r+\ell)dr dz)}
    \\&
    \les_p \| \omega \|_{L^p((r+\ell)\, dr \, dz)} 
    \les_{p} \ell^{1/p}\| \omega \|_{L^2_\gamma} \, .
    \end{split}
\end{equation}
for any $p \in (1,3/2)$, where $1/p^* = 1/p-1/3$ and  $1/q= 1/p-2/3$.

To bound $I$, it is enough to use H{\"o}lder's inequality and \eqref{eqn:CZ} with $p=6/5$ and $q = 6$. Indeed,
\begin{equation}
\begin{split}
\label{eqn:boh}
\|(r+\ell)^{-1} \psi_\ell \|_{L^{2}(B_{\bar R})}
&\les
\|(r+\ell)^{-1} \psi_\ell \|_{L^{6}(B_{\bar R})}
 \les \ell^{-1} \|\psi_\ell \|_{L^{6}((r+\ell) \, dr \, dz )}
\\&\les
  \ell^{-1/6} \| \omega \|_{L^2_\gamma} \, .
 \end{split}
\end{equation}

Let us now study $II$.
\begin{equation}\label{eqn:CZ2}
\| \psi \|_{\dot C^{2-2/p}}+
   \| \nabla \psi \|_{L^{p^*}(  dr \, dz )}\les_p \| \omega \|_{L^p( dr \, dz)} \les_p \| \omega \|_{L^2_\gamma}
\end{equation}
for any $p\in (1,2)$, where $1/p^*= 1/p-1/2$.
Notice that $\psi$ is only defined up to a constant; here and in the following, we assume that $\psi(0)=0$, which implies
\begin{equation}\label{zzz5}
    \| \psi \|_{L^q(B_R)} \les_p R^{2+\frac{2}{q}-\frac{2}{p}} \| \omega \|_{L^2_\gamma} \, 
    \qquad \text{for any $R>1$, $q\ge 1$, $p\in (1,2)$} \, .
\end{equation}


Let us denote by $\chi_\ell$ a standard cutoff which is $1$ on $ B_{\ell/4}$ and $0$ outside $B_{\ell/2}$, with $|\nabla \chi_\ell| \les \ell^{-1}$ and $|\nabla ^2\chi_\ell| \les \ell^{-2}$. Here, $\ell \geq 8\bar{R}$.
We write an equation for $(\psi_\ell-\psi)\chi_\ell$ by using the equations for $\psi_\ell$ and $\psi$. We obtain
\begin{equation}\label{zzz4}
    \Delta [(\psi_\ell-\psi) \chi_\ell] = \frac{1}{(r+\ell)^2} \psi_\ell \chi_\ell -  \frac{1}{r+\ell} \partial_r\psi_\ell \chi_\ell + 
2\nabla (\psi_\ell-\psi) \cdot \nabla \chi_\ell+ (\psi_\ell-\psi) \Delta \chi_\ell \, .
\end{equation}
By using H{\"o}lder's inequality and the Sobolev embedding in $\R^2$, we deduce
\begin{equation}
\| \nabla (\psi_\ell-\psi)\|_{L^2 (B_{\bar R})}
 \les
\| \nabla [(\psi_\ell-\psi) \chi_\ell]\|_{L^6}
 \les
 \|  \Delta [(\psi_\ell-\psi) \chi_\ell] \|_{L^{3/2}} \, ,
\end{equation}
and \eqref{zzz4} implies
\begin{equation}
\begin{split}
 &\| \Delta  [(\psi_\ell-\psi) \chi_\ell] \|_{L^{3/2}(B_{\ell/2})}
\\& =
\|(r+\ell)^{-2} \psi_\ell \chi_\ell - (r + \ell)^{-1} \partial_r\psi_\ell \chi_\ell +
2\nabla (\psi_\ell-\psi) \cdot \nabla \chi_\ell+ (\psi_\ell-\psi) \Delta \chi_\ell \|_{L^{\frac3 2}(B_{\ell/2})}
\\
&\les
\frac{1}{\ell^2} \big(\| \psi_\ell \|_{L^{3/2}(B_{\ell/2})}+\| \psi \|_{L^{3/2}(B_{\ell/2})}\big) 
\\ & \qquad\qquad
+\frac 1 \ell \big(\|(\p_r + (r+\ell)^{-1}, \p_z) \psi_\ell \|_{L^{\frac 32}(B_{\ell/2})}+ \|\nabla \psi \|_{L^{\frac 32}(B_{\ell/2})}\big) \, .
\end{split}
\end{equation}
By using H{\"o}lder's inequality and \eqref{eqn:CZ} we estimate the first term:
\begin{equation}
    \frac{1}{\ell^2} \| \psi_\ell \|_{L^{3/2}(B_{\ell/2})}
    \les 
    \frac{1}{\ell} \| \psi_\ell \|_{L^{4}((r+\ell) \, dr\, dz)}
    \les
    \frac{1}{\ell^{\frac{1}{12}}} \| \omega \|_{L^2_\gamma} \, .
\end{equation}
To estimate the second term we employ \eqref{zzz5}:
\begin{equation}
    \frac{1}{\ell^2} \| \psi \|_{L^{3/2}(B_{\ell/2})}
    \les \frac{1}{\ell^{1/2}}  \| \omega \|_{L^2_\gamma} \, .
\end{equation}
We apply \eqref{eqn:CZ} to the third term: For $\eps>0$ small enough, it holds
\begin{equation}
    \begin{split}
    \frac 1 \ell \|(\p_r + (r+\ell)^{-1}, \p_z) \psi_\ell \|_{L^{\frac 32}(B_{\ell/2})}
     & \les
     \frac {1}{\ell^{9/10}} \|(\p_r + (r+\ell)^{-1}, \p_z) \psi_\ell \|_{L^{\frac{3}{2} + \eps}(B_{\ell/2})}
     \\& \les 
    \frac 1 {\ell^{4/3}} \|(\p_r + (r+\ell)^{-1}, \p_z) \psi_\ell \|_{L^{\frac{3}{2} + \eps}((r+\ell)dr dz)}
    \\&\les_\eps
    \frac{1}{\ell^{1/3}} \| \omega \|_{L^2_\gamma} \, .
    \end{split}
\end{equation}
To estimate the last term, we use H{\"o}lder's inequality and \eqref{eqn:CZ2}:
\begin{equation}
   \frac{1}{\ell} 
   \| \nabla \psi\|_{L^{3/2}(B_{\ell/2})}
   \les 
   \frac{1}{\ell^{\frac{7}{15}}} \| \nabla \psi \|_{L^{5/2}(B_{\ell/2})}
   \les \frac{1}{\ell^{\frac{7}{15}}}\| \omega \|_{L^2_\gamma} \, .
\end{equation}





\textbf{3. Conclusion}. From the above arguments and Lemma~\ref{lem:auxlem}, we conclude the following. Let $K$ be a compact subset of $\rho(\M_\infty + \K_\infty) \cap \{ \Re \lambda > 0 \}$. There exists $\ell_1 \geq \ell_0$ such that, for all $\ell \geq \ell_1$, the operator $\lambda - \M_\ell P_\ell - \K_\ell P_\ell - \S_\ell P_\ell$ is invertible on $L^2_\gamma$, and its inverse converges to $R(\lambda, \M_\infty + \K_\infty)$ pointwise, uniformly on $K$, as $\ell \to +\infty$.


Let $\vec{c} \subset B_{\varepsilon}(\ell_\infty)$ be a counterclockwise, smooth, simple, closed curve containing $\lambda_\infty$ in its interior and no other element of $\sigma(\L_\infty)$. Let $K = \vec{c}$. When $\ell \geq \ell_1$, we may define the spectral projections $\Pr_\ell : L^2_{\gamma} \to L^2_{\gamma}$ according to
\begin{equation}
    \Pr_\ell = \frac{1}{2\pi i} \int_{\vec{c}} R(\lambda,\L_{\ell}P_\ell) \, d\lambda,
\end{equation}
and analogously $\Pr_\infty$. In particular, the integrand is well defined when $\ell \geq \ell_1$. Then $\Pr_\ell f \to \Pr_\infty f$ for all $f \in C^\infty_0(\R^2)$ as $k \to +\infty$. In particular, since $\Pr_\infty$ is non-trivial, we must have that $\Pr_\ell$ is non-trivial for all sufficiently large $\ell$. \end{proof}


\section{Linear instability: from Euler to Navier-Stokes equations}
\label{sec:eulertonavierstokes}

For sufficiently large $\ell$ (which will be fixed momentarily), we define $\bar{U}$ to be the axisymmetric-no-swirl velocity field on $\R^3$ associated with $\tilde{u}_\ell$:
\begin{equation}
    \bar{U}(x) = \tilde{u}_\ell^r(r+\ell,z) e_r + \tilde{u}_\ell^z(r+\ell,z) e_z,
\end{equation}
where $x = (x',z)$ and $r = |x'|$. Let $L^2_{\rm aps}$ denote the set of $L^2$-integrable `axisymmetric pure-swirl' vector fields $\Omega = \Omega^\theta(r,z) e_\theta$.

In this section, we concern ourselves with instability for the operator
\begin{equation}
    \L^{(\beta)}_{\rm vor} : D(\L^{(\beta)}_{\rm vor}) \subset L^2_{\rm aps} \to L^2_{\rm aps}
\end{equation}
\begin{equation}
    - \L^{(\beta)}_{\rm vor} = (-1-\frac{\xi}{2} \cdot \nabla_\xi) \Omega - \Delta \Omega + \beta [\bar{U},\Omega] + \beta [U,\bar{\Omega}],
\end{equation}
where $U = \BS_{3d}[\Omega]$. Its domain is
\begin{equation}
    D(\L^{(\beta)}_{\rm vor}) := \{ \Omega \in L^2_{\rm aps} : \Omega \in H^2(\R^3), \, \xi \cdot \nabla_\xi \Omega \in L^2(\R^3) \} \, .
\end{equation}
We further introduce the operators comprising $\L^{(\beta)}_{\rm vor}$:
\begin{equation}
- \D \Omega := - \frac{3}{4} \Omega - \frac{\xi}{2} \cdot \nabla_\xi \Omega - \Delta \Omega \,
\end{equation}
\begin{align}
- \M \Omega &:= \bar{U} \cdot \nabla \Omega \label{e:S1}
\end{align}
\begin{equation}
    - \S \Omega :
= 
 - \bar \Omega \cdot \nabla U - \Omega \cdot \nabla \bar U 
\end{equation}
\begin{align}
- \K \Omega &:= (\BS_{3d}[\Omega] \cdot \nabla) \bar \Omega \label{e:compatto}\, .
\end{align}
The naming convention is as follows: $\D$ is the `dissipation', $\M$ is the `main term', $\S$ arises from the `stretching' and is `small', and $\K$ is easily verified to be `compact'. Notice that $\D - \Delta$ is a skew-symmetric operator; this will play a role in Lemma~\ref{lemma:onlynew}.

Finally, we define
 \begin{equation}
    \T_\beta: =  \frac{1}{\beta} \D + \M + \S + \K =  \frac{1}{\beta} \left( \L^{(\beta)}_{\rm vor} - \frac{1}{4} \right).
\end{equation}
In analogy with $\T_\beta$ we introduce
\begin{equation}
    \T_\infty := \M + \S + \K \, ,
\end{equation}
whose domain is $D(\T_\infty) = \{ \Omega \in L^2_{\rm aps} : \bar{U} \cdot \nabla \Omega \in L^2(\R^3) \}$. 

Let
    \begin{equation}
        a := \sup_{\lambda \in \sigma(\T_\infty)} \Re \lambda, \quad \mu := \| S \|_{L^2_{\rm aps} \to L^2}.
    \end{equation}
     The above constants depend on $\ell$, which we fix momentarily to ensure that
     \begin{equation}
        \label{eq:abiggerthanmu}
         a > \mu.
     \end{equation} From Proposition~\ref{pro:axisymmetricinstab}, we know that $a > \text{const}.$ for all sufficiently large $\ell$. On the other hand, we now justify that $\mu \to 0$ as $\ell \to +\infty$. Using the axisymmetric-pure-swirl structure to estimate~$\S$, we have
\begin{equation}
    (\S \Omega)^{\theta}=
 \frac 1 r \bar \Omega^\theta  U^r +  \frac 1 r \Omega^\theta  \bar U^r.
\end{equation}
Since $\bar{U}$ is compactly supported and, specifically, $\supp \bar{U} \subset \{ x \in \R^3 : |r - \ell|^2 + |z|^2 \leq \bar{R}^2 \}$, we have by the Calder{\'o}n--Zygmund estimates that
\begin{equation}
    \begin{split}
        \|\S \Omega \|_{L^2} &\leq \left\|\frac 1 r \bar \Omega^\theta  U^r \right\|_{L^2(r\, dr\, dz)} + \left\|\frac 1 r \Omega^\theta  \bar U^r \right\|_{L^2(r\, dr\, dz)}
\\&\les 
\frac 1 {\ell} \| U^r \|_{L^6(r\, dr\, dz)} \| \bar \Omega^\theta \|_{L^3(r\, dr\, dz)}
+ \frac 1 {\ell} \| \bar U^r \|_{L^\infty} \| \Omega^\theta \|_{L^2(r\, dr\, dz)} 
\\&\les 
\frac {1}{\ell^{2/3}} \|\Omega^\theta \|_{L^2(r\, dr\, dz)},
    \end{split}
\end{equation}
where the implicit constant may depend on $\| \bar \Omega^\theta\|_{L^\infty}$.

From now on, we consider $\ell$ and the background $\bar{U}$ as fixed and satisfying~\eqref{eq:abiggerthanmu}.

\begin{theorem}[Self-similar Navier-Stokes instability: vorticity]\label{thm:ns-inst}
Let $\lambda_\infty$ be an unstable eigenvalue of $\T_\infty$ with $\Re \lambda_\infty > \mu$. For all $\varepsilon \in (0,\Re \lambda_\infty - \mu)$, there exists $\beta_0 > 0$ such that, for all $\beta \geq \beta_0$, 
$\T_\beta$ has an unstable eigenvalue $\lambda_\beta$ satisfying $|\lambda_\beta - \lambda_\infty| < \varepsilon$, and  $\L^{(\beta)}_{\rm vor}$ has an unstable eigenvalue $\tilde{\lambda}_\beta$ defined by
\begin{equation}
    \label{eq:eigenvalueformula}
    \tilde{\lambda}_\beta = \beta \lambda_\beta + \frac{1}{4}.
\end{equation}
\end{theorem}



To prove nonlinear instability in Section~\ref{sec:nonlinearinstability}, it will be more convenient to work with the velocity formulation. Therefore, we introduce
\begin{equation}
    \L^{(\beta)}_{\rm vel} : D(\L^{(\beta)}_{\rm vel}) \subset L^2_\sigma(\R^3) \to L^2_\sigma(\R^3)
\end{equation}
\begin{equation}
    - \L^{(\beta)}_{\rm vel} U := (-\frac{1}{2} - \frac{\xi}{2} \cdot \nabla_\xi) U - \Delta U + \bP (\bar{U} \cdot \nabla U + U \cdot \nabla \bar{U})
\end{equation}
whose domain is
\begin{equation}
    D(\L^{(\beta)}_{\rm vel}) := \{ U \in L^2_\sigma(\R^3) : U \in H^2(\R^3), \, \xi \cdot \nabla_\xi U \in L^2(\R^3) \} \, .
\end{equation}
Here, $L^p_\sigma(\R^3)$ is the space of divergence-free $L^p$-integrable vector fields on $\R^3$.
Notice that we do not impose axisymmetry on $\L^{(\beta)}_{\rm vel}$ .

\begin{corollary}[Self-similar Navier-Stokes instability: velocity]
    \label{cor:velocityinstability}
For $\beta \geq \beta_0$, the unstable eigenvalue $\tilde{\lambda}_\beta$ in Theorem~\ref{thm:ns-inst} is also an unstable eigenvalue of $\L^{(\beta)}_{\rm vel}$.
\end{corollary}

The operator $\L_{\rm ss}$ in the next section will be defined as $\L^{(\beta)}_{\rm vel}$ for sufficiently large~$\beta$.



Before we prove Theorem~\ref{thm:ns-inst} and Corollary~\ref{cor:velocityinstability}, we prove two preliminary lemmas.


%

\begin{lemma}\label{lemma:onlynew}
For all $\beta > 0$ and $\lambda \in \{ \Re \lambda > \mu \}$, we have
\begin{equation}\label{e:uniform-bound-inverse-beta}
\| R(\lambda, \beta^{-1} \D + \M + \S) \|_{L^2_{\rm aps} \to L^2_{\rm aps}} \leq 
(\Re \lambda - \mu)^{-1},
\end{equation}
and, for all $\Omega \in L^2_{\rm aps}$, we have  
\begin{equation}
    R(\lambda, \beta^{-1} \D + \M + \S) \Omega \overset{\beta \to +\infty}{\longrightarrow} R(\lambda,  \M + \S) \Omega \text{ in } L^2
\end{equation}
locally uniformly in $\{ \Re \lambda > \mu \}$.
\end{lemma}

From~\eqref{e:uniform-bound-inverse-beta} and the compactness of $\K$, we know that the essential spectrum of $\T_\beta$ is contained in the left half-plane $\{ \Re \lambda \leq \mu \}$, and the remainder of its spectrum is discrete, consisting of isolated eigenvalues of finite algebraic multiplicity. 

\begin{proof}
Consider the advection-diffusion equation
\begin{equation}
    \label{eq:timedependentadvectiondiffusion}
\p_\tau \Omega^\beta - \frac{1}{\beta} (\Delta + \frac{3}{4} + \frac{\xi}{2} \cdot \nabla_\xi) \Omega^\beta + \bar{U} \cdot \nabla \Omega^\beta - \S \Omega^\beta = 0
\end{equation}
and its inviscid counterpart
\begin{equation}
    \label{eq:inviscidcounterpart}
    \p_\tau \Omega + \bar{U} \cdot \nabla \Omega - \S \Omega = 0
\end{equation}
supplemented with the initial condition 
$\Omega^\beta(\cdot,0) = \Omega(\cdot,0) = \Omega_0 \in C^\infty_0(\R^3) \cap L^2_{\rm aps}$. The basic energy estimate for solutions to~\eqref{eq:timedependentadvectiondiffusion} is
\begin{equation}
    \frac{1}{2} \frac{d}{d\tau} \| \Omega^\beta \|_{L^2}^2 + \frac{1}{\beta} \int |\nabla \Omega^\beta|^2 \, dx 
    \leq \| \S \|_{L^2 \to L^2} \| \Omega^\beta \|_{L^2}^2 \leq \mu \| \Omega^\beta \|_{L^2}^2.
\end{equation}
Therefore, we have
\begin{equation}
        \frac{1}{2} \frac{d}{d\tau} \| e^{\tau \mu} \Omega^\beta(\cdot,\tau) \|_{L^2}^2 + 
        \frac{1}{\beta} \int |\nabla e^{\tau \mu} \Omega^\beta(\cdot, \tau)|^2 \, dx  \leq 0.
\end{equation}
Hence,
\begin{equation}
    \label{eq:growthestimate}
    \| \Omega^\beta\|_{L^2} \leq e^{\tau \mu} \| \Omega_0 \|_{L^2}.
\end{equation}
The growth estimate~\eqref{eq:growthestimate} ensures that the Laplace transform
\begin{equation}
    R(\lambda, \beta^{-1} \D + \M + \S) \Omega_0 = \int_0^\infty e^{-s \lambda} \Omega^\beta(\cdot,s)  \, ds
\end{equation}
is well defined and satisfies the desired estimate~\eqref{e:uniform-bound-inverse-beta} when $\Re \lambda > \mu$.\footnote{One can also see the estimate~\eqref{e:uniform-bound-inverse-beta} directly from the resolvent advection-diffusion PDE.} Additionally, we have the analogous Eulerian energy estimates for $\Omega$ without the term corresponding to $\beta^{-1} \D$. The estimates can also be extracted from the ODE along Lagrangian trajectories:
\begin{equation}
    \label{eq:lagrangiantrajectories}
    \frac{d}{d\tau} \left( \Omega \circ X_\tau \right) = (\S \Omega) \circ X_\tau,
\end{equation}
where $X_\tau$ is the flow map associated to $\bar{U}$.
It is moreover possible to demonstrate that $\Omega$ is smooth, though we only require $H^1$ estimates, and compactly supported. The smoothness can be extracted from Eulerian energy estimates for $\p_i \Omega$ or from the ODE~\eqref{eq:lagrangiantrajectories}. The compact support properties follow from~\eqref{eq:lagrangiantrajectories} and the compact support property of $\S$, namely, $\supp \S \Omega \subset \supp \bar{U}$.\footnote{Notably, the $H^1$ norm of $\Omega$ may grow more quickly than $e^{\tau \mu}$; this can be seen in the term $\p_i \bar{U} \cdot \nabla \Omega$ in the equation for $\p_i \Omega$. This is why we avoid estimating the gradient of solutions to the resolvent problem.}

We now demonstrate the inviscid limit $\Omega^\beta \to \Omega$  for fixed  initial data $\Omega_0 \in C^\infty_0(\R^3) \cap L^2_{\rm aps}$ with rate of convergence possibly depending on $\Omega_0$. The equation satisfied by the difference $\Omega^\beta - \Omega$ is
\begin{equation}
    \label{eq:equationfordifference}
    \begin{aligned}
&\p_\tau ( \Omega^\beta - \Omega ) - \frac{1}{\beta} (\Delta + \frac{3}{4} + \frac{\xi}{2} \cdot \nabla_\xi) ( \Omega^\beta - \Omega ) + \bar{U} \cdot \nabla ( \Omega^\beta - \Omega ) - \S ( \Omega^\beta - \Omega ) \\
&\quad = \frac{1}{\beta} F := \frac{1}{\beta} (\Delta + \frac{3}{4} + \frac{\xi}{2} \cdot \nabla_\xi) \Omega,
\end{aligned}
\end{equation}
where the right-hand side is an error term to be controlled below. Since
\begin{equation}
    \| F \|_{L^2_\tau H^{-1}_\xi(\R^3 \times (0,T))} = \left\| (\Delta + \frac{3}{4} + \frac{\xi}{2} \cdot \nabla_\xi) \Omega \right\|_{L^2_\tau H^{-1}_\xi(\R^3 \times (0,T))} \leq C(T,\Omega_0),
\end{equation}
the energy estimate for~\eqref{eq:equationfordifference} gives
\begin{equation}
\begin{aligned}
        &\frac{1}{2} \frac{d}{d\tau} \| \Omega^\beta - \Omega \|_{L^2}^2 + \frac{1}{\beta} \int |\nabla (\Omega^\beta - \Omega)|^2 \, dx  \\
        &\quad \leq \mu \| \Omega^\beta - \Omega \|_{L^2}^2 + \frac{1}{\beta} \la F, \Omega^\beta - \Omega \ra \\
        &\quad \leq \mu \| \Omega^\beta - \Omega \|_{L^2}^2 + \frac{C}{\beta} \| F(\cdot,\tau) \|_{H^{-1}}^2 + \frac{1}{2\beta} \left( \| \Omega^\beta - \Omega \|_{L^2}^2 + \| \nabla (\Omega^\beta - \Omega) \|_{L^2}^2 \right).
        \end{aligned}
\end{equation}
The term involving $\nabla (\Omega^\beta - \Omega)$ on the right-hand side can be absorbed into the left-hand side, from which we conclude that, for all finite $T > 0$, we have
\begin{equation}
    \label{eq:convergence}
    \| \Omega^\beta - \Omega \|_{L^\infty_\tau L^2_\xi(\R^3 \times (0,T))} \overset{\beta \to +\infty}{\longrightarrow} 0,
\end{equation}
and convergence is with rate $O(\beta^{-1})$ depending on $T$ and $\Omega_0$.

To complete the proof of uniform convergence for $\Omega_0 \in C^\infty_0(\R^3) \cap L^2_{\rm aps}$, we estimate
\begin{equation}
    \begin{aligned}
        &\| [ R(\lambda, \beta^{-1} \D + \M + \S) - R(\lambda, \M + \S) ] \Omega_0 \|_{L^2} \\
        &\quad \leq \int_0^{T} e^{-s \lambda} \| (\Omega^\beta - \Omega)(\cdot,s) \|_{L^2} \, ds \\
            &\quad\quad + \int_T^{+\infty} e^{-s\lambda} \| \Omega^\beta(\cdot,s) \|_{L^2} \, ds + \int_T^{+\infty} e^{-s\lambda} \| \Omega(\cdot,s) \|_{L^2} \, ds \\
            &\quad \leq o_{\beta \to +\infty}(1) + 2e^{-T(\Re \lambda - \mu)},
    \end{aligned}
\end{equation}
where $o(1)$ may depend on $T$ and $\Omega_0$. The proof of convergence for general $\Omega_0 \in L^2_{\rm aps}$ follows from density.
\end{proof}

The following lemma follows from the previous one via abstract arguments: 
\begin{lemma}\label{l:two}
For all compact subsets $K$ of $\{ \Re \lambda > \mu \} \setminus \sigma(\T_{\infty})$, there exists $\beta_0 > 0$, depending on $K$, such that, for all $\beta \geq \beta_0$, the operator $\T_\beta$ is invertible,
\begin{equation}
    \sup_{\beta \geq \beta_0} \sup_{\lambda \in K} \| R(\lambda,\T_\beta) \|_{L^2_{\rm aps} \to L^2} < +\infty,
\end{equation}
and, for all $\Omega \in L^2_{\rm aps}$, we have
\begin{equation}\label{e:strong_convergence}
R(\lambda,\T_\beta) \Omega \overset{\beta \to +\infty}{\longrightarrow} R(\lambda,\T_\infty) \Omega \text{ in } L^2
\end{equation}
uniformly-in-$\lambda$ on $K$.
\end{lemma}


\begin{proof}
This is analogous to the proof of Lemma~\ref{lem:auxlem}. Recall that
\begin{equation}
    \lambda - \T_\beta = (\lambda - \beta^{-1} \D - \M - \S) [I - R(\lambda, \beta^{-1} \D + \M + \S) \K].
\end{equation}
Since $\K$ is a compact operator on a separable Hilbert space, we may approximate it in operator norm by a sequence of finite-rank operators $(F_k)$. The uniform convergence in Lemma~\ref{lemma:onlynew} implies that, for each $k \in \N$, the operator $R(\lambda, \beta^{-1} \D + \M + \S) F_k$ converges in operator norm to $R(\lambda, \M + \S)F_k$, uniformly-in-$\lambda$ on $K$; hence, $R(\lambda, \beta^{-1} \D + \M + \S) \K \to R(\lambda, \M  + \S)\K$ similarly. 
Therefore, for $\beta$ sufficiently large, the operator $I - R(\lambda, \beta^{-1}\D + \M + \S) \K$ is a small perturbation of the invertible operator $I - R(\lambda, \M + \S) \K$, so it is invertible. The convergence follows straightforwardly. We leave further details to the reader.
\end{proof}

\begin{proof}[Proof of Theorem~\ref{thm:ns-inst}]
This is analogous to the conclusion of the proofs of Propositions~\ref{pro:truncatedunstablevortex} and~\ref{pro:axisymmetricinstab}. We consider a counterclockwise, simple, smooth, closed curve $\vec{c} \subset B_\varepsilon(\lambda_\infty) \cap \rho(\T_\infty)$ containing $\lambda_\infty$ in its interior and no other element of $\sigma(\T_\infty)$. We define the spectral projection $\Pr_\beta$ according to
\begin{equation}
    \Pr_\beta = \frac{1}{2\pi i} \int_{\vec{c}} R(\lambda,\T_\beta) \, d\lambda.
\end{equation}
We use the convergence in Lemma~\ref{l:two} to determine that, for $\beta$ sufficiently large, the projection is non-trivial. We leave further details to the reader. \end{proof}

\begin{proof}[Proof of Corollary~\ref{cor:velocityinstability}]
By choosing $\beta$ sufficiently large, we may ensure that $\L^{(\beta)}_{\rm vor}$ has an unstable eigenvalue, denoted simply by $\lambda$, with $\Re \lambda \geq 1$, see the formula~\eqref{eq:eigenvalueformula}. We claim that an $L^2$-eigenfunction $\Omega$ of $\L^{(\beta)}_{\rm vor}$ associated to such an eigenvalue further satisfies $\Omega \in L^1(\R^3)$. Once this is known, we have that the velocity field $U = \BS_{3d}[\Omega]$ belongs to $L^2(\R^3)$ by standard mapping properties and moreover is an unstable eigenfunction of $\L^{(\beta)}_{\rm vel}$ with eigenvalue $\lambda$.

Notice that $\Omega \in D(\T_\beta)$ solves the equation
\begin{equation}\label{eqn:omega}
   \lambda \Omega - (1+\frac{\xi}{2} \cdot \nabla_\xi) \Omega - \Delta \Omega =F
\end{equation}
where $- F= \beta [\bar{U},\Omega] +\beta [U,\bar{\Omega}]$ belongs to $L^2(\R^3)$ and is compactly supported. 
{The solution of~\eqref{eqn:omega} can be found explicitly by `undoing' the similarity variables: Define
\begin{equation}
    h(x,t) = t^{\lambda-1} \Omega \left(\frac x {t^{1/2}}\right), \quad 
    g(x,t) = t^{\lambda-2} F\left(\frac x {t^{1/2}}\right).
\end{equation}
Then
\begin{equation}
    \label{eq:definitionofh}
    \p_t h - \Delta h = g, \quad h(\cdot,0) \equiv 0.
\end{equation}
The zero initial condition in~\eqref{eq:definitionofh} comes from $\| h(\cdot,t)\|_{L^2} = t^{\Re \lambda-\frac{1}{4}} \| \Omega \|_{L^2(\R^3)} \to 0$ as $t \to 0^+$. Notice also
\begin{equation}
    \label{eq:wheregbelongs}
    \| g(\cdot,t) \|_{L^1(\R^3)} = t^{\Re \lambda - \frac{1}{2}} \| F \|_{L^1(\R^3)}.
\end{equation}
Hence, $g \in L^\infty_t L^1_x(\R^3 \times (0,1))$. We can represent the eigenfunction $\Omega$ as
\begin{equation}
    \Omega(x) = h(x,1) = \int_0^1 e^{\Delta(1-s)} g(\cdot,s) \, ds \in L^1(\R^3),
\end{equation}
due to~\eqref{eq:wheregbelongs} and the mapping property $\| e^{\Delta(1-s)} \|_{L^1 \to L^1} \leq 1$.}
\end{proof}

A similar argument can be used to characterize the domains of $\L^{(\beta)}_{\rm vor}$ and $\L^{(\beta)}_{\rm vel}$, see~\cite[Section 2]{jiasverakillposed}.

\section{Nonlinear instability}
\label{sec:nonlinearinstability}


In this section, we demonstrate how to use the linear instability proven in Corollary~\ref{cor:velocityinstability} to construct non-unique solutions to the forced Navier-Stokes equations.

We employ the uppercase/lowercase notation introduced in~\eqref{eq:definitionofuandU} for functions in similarity/physical variables, respectively. Recall that $L^p_\sigma$ is the space of divergence-free vector fields in $L^p(\R^3)$.


\begin{theorem}[Non-uniqueness criterion]\label{thm:nonuniqueness criterion}
Let $\bar{U}$ be a compactly supported, smooth, divergence-free vector field on $\R^3$. Suppose that~\ref{item:maintheorema} in Theorem~\ref{thm:refined} holds, that is, the linearized operator $\L_{\rm ss} \: D(\L_{\rm ss}) \subset L^2_\sigma \to L^2_\sigma$ defined by
\begin{equation}
     - \L_{\rm ss} U = - \frac{1}{2} \left( 1 + \xi \cdot \nabla_\xi \right) U - \Delta U + \bP (\bar{U} \cdot \nabla U + U \cdot \nabla \bar{U}) \, ,
\end{equation}
where $D(\L_{\rm ss}) := \{ U \in L^2_\sigma : U \in H^2(\R^3), \,  \xi \cdot \nabla U \in L^2(\R^3) \}$,
has an unstable eigenvalue~$\lambda$. Then $\lambda$ can be chosen to be maximally unstable, that is, $\Re \lambda = \sup_{z \in \sigma(\L_{\rm ss})} \Re z > 0$,  and moreover,~\ref{item:maintheoremb} in Theorem~\ref{thm:refined} holds.
\end{theorem}

As mentioned above~\eqref{eq:satisfytheasymptotics}, a second solution to~\eqref{eq:ns} is sought as a trajectory on the unstable manifold of $\bar{U}$ associated to the most unstable eigenvalues of $\L_{\rm ss}$. We make the following ansatz:
\begin{equation}
    U = \bar{U} + U^{\rm lin} + U^{\rm per},
\end{equation}
where
\begin{equation}
U^{\rm lin}(\cdot,\tau) = \Re (e^{\lambda \tau} \eta) \, 
\end{equation}
is a solution of the linearized PDE $\p_\tau U^{\rm lin} = \L_{\rm ss} U^{\rm lin}$,
$\lambda = a + i b$ is an unstable eigenvalue of $\L_{\rm ss}$ with maximal growth rate $a > 0$, and $\eta\in L^2_\sigma$ is a non-trivial eigenfunction associated to $\lambda$. The remainder $U^{\rm per}$, whose equation is~\eqref{eq:nonlin2} below, decays in $L^2$ as $o(e^{\tau a})$ when $\tau \to -\infty$. The additional decay ensures that $U$ is not identical to $\bar{U}$ and, moreover, guarantees non-uniqueness.

In Section~\ref{sec:nonlinearconstruction}, we employ a fixed point argument to solve~\eqref{eq:nonlin2} and construct $U^{\rm per}$. First, we introduce the semigroup generated by $\L_{\rm ss}$.

\subsection{The linear operator}
In this section we study $\sigma(\L_{\rm ss})$, the spectrum of the closed, densely defined operator $\L_{\rm ss} : D(\L_{\rm ss}) \subset L^2_\sigma \to L^2_\sigma$, and the associated semigroup $e^{\tau \L_{\rm ss}} : L^2_\sigma \to L^2_\sigma$.
The latter is well defined and strongly continuous, as proven in \cite[Section 2]{jiasverakillposed}.

We define the \textit{spectral bound} of $\L_{\rm ss}$ as
\begin{equation}
	s(\L_{\rm ss}) 
	:= \sup\{  \Re \lambda\, : \, \lambda \in \sigma(\L_{\rm ss}) \} \, ,
\end{equation}
which is bounded by the \textit{growth bound}
\begin{equation}
\omega_0(\L_{\rm ss}) := \inf \{\omega \in \R \, : \,  \| e^{\tau \L_{\rm ss}}\|_{L^2_\sigma \to L^2_\sigma} \le M(\omega)e^{\tau \omega} \}
\end{equation}
of the semigroup. We refer the reader to \cite[Proposition 2.2]{EngelNagel} for the proof of the inequality $s(\L_{\rm ss})\le \omega_0(\L_{\rm ss})$. 
As a consequence of \cite[Corollary 2.11]{Katobook}, in our generality, we have the following.

\begin{lemma}
    It holds $\omega_0(\L_{\rm ss}) = \max \{ \omega_{\rm ess}(\L_{\rm ss}), s(\L_{\rm ss}) \}$, where
    \begin{equation}
        \omega_{\rm ess}(\L_{\rm ss}):= \inf_{\tau>0} \frac{1}{\tau}
        \log \left( \inf  \{ \| e^{\tau \L_{\rm ss}} - K\|_O \, : \, K  \, \, {\rm is \, \, compact} \}\right).
    \end{equation}
    Moreover, for every $w>\omega_{\rm ess}(\L_{\rm ss})$, the set $\sigma(\L_{\rm ss})\cap \{\Re \lambda> w\}$ is finite, and the corresponding spectral projection has finite rank.
\end{lemma}
In other words, the identity $s(\L_{\rm ss}) = \omega_0(\L_{\rm ss})$ holds up to taking into account contributions coming from the essential spectrum.

Applying \cite[Lemma 2.7]{jiasverakillposed}, we have that $\omega_{\rm ess}(\L_{\rm ss}) \le -1/4$, and hence the following result holds.

\begin{proposition}\label{prop:specLOmega}
Assume $a := s(\L_{\rm ss})>0$. Then $a<\infty$, and there exist $\lambda = a + ib\in \sigma(\L_{\rm ss})$ and $\eta = \eta_1 + i \eta_2$ such that $\eta_1, \eta_2 \in D(\L_{\rm ss})$ and $\L_{\rm ss} \eta = \lambda \eta$.

Moreover, for any $\delta>0$, it holds
\begin{equation}\label{eq: growth est}
    \| e^{ \tau \L_{\rm ss}} U_0 \|_{L^2} \le M(\delta) e^{\tau( a + \delta)} \| U_0 \|_{L^2} \, ,
    \qquad \forall\,  U_0 \in L^2_\sigma \, .
\end{equation}
\end{proposition}

A more elementary way to prove that $\sigma(\L_{\rm ss}) \cap \{ \Re \lambda \geq \mu \}$ is countable for all $\mu > - 1/4$ was already mentioned in Section~\ref{sec:spectralprelim}, since $U \mapsto \bP (\bar{U} \cdot \nabla U + U \cdot \nabla \bar{U})$ is relatively compact with respect to $U \mapsto (1+\xi \cdot \nabla_\xi)U/2 + \Delta U$, whose essential spectrum is contained in $\{ \Re \lambda \leq - 1/4 \}$, as is easily demonstrated by an energy estimate, see~\cite[Section 2]{jiasverakillposed}. Proposition~\ref{prop:specLOmega} contains more information, and in particular, the growth estimate~\eqref{eq: growth est}. 

   We conclude this section by proving that $e^{\tau \L_{\rm ss}}$ enjoys parabolic regularity estimates with sharp growth at infinity.
   
   \begin{lemma}\label{prop:semigroup}
	Assume $a = s(\L_{\rm ss})>0$. Then, for any $\sigma_2 \geq \sigma_1 \geq 0$ and $\delta > 0$, it holds
	\begin{equation}\label{eq:reg}
		\| e^{ \tau \L_{\rm ss}} U_0 \|_{H^{\sigma_2}} \le \frac{M(\sigma_1,\sigma_2,\delta)}{\tau^{(\sigma_2 - \sigma_1)/2}} e^{\tau (a + \delta)} \| U_0 \|_{H^{\sigma_1}} \, ,
	\end{equation}
	for any $U_0 \in L^2_\sigma \cap H^{\sigma_1}(\R^3)$.
\end{lemma}

\begin{proof}
It suffices to treat the case when $\sigma_1,\sigma_2$ are non-negative integers. The general case follows by interpolation.

We divide the proof into two steps. First, we show smoothing estimates~\eqref{eq:reg} for small times, say, $0<\tau \le 2$, and afterward we address the long-time behaviour.

{\bf Step 1}. For any $0\le m \le k$, it holds
    \begin{equation}\label{eq:reg1}
		\| e^{ \tau \L_{\rm ss}} U_0 \|_{H^k} \le M(k,m) \tau^{-\frac{k-m}{2}} \| U_0 \|_{H^{m}} \, ,
		\quad \forall\, U_0 \in L^2_\sigma \cap H^m(\R^3), \, \, \tau \in (0,2) \, .
	\end{equation}

To ease notation, we set $U(\cdot,\tau):=  e^{ \tau \L_{\rm ss}} U_0$.
We study the problem in physical variables: Setting
\begin{equation}
    u(x, t) := \frac{1}{\sqrt{t+1}} U\left( \frac{x}{\sqrt{t+1}}, \log(t+1)\right) \, , \quad
    \bar u(x,t) := \frac{1}{\sqrt{t+1}} \bar  U\left( \frac{x}{\sqrt{t+1}}\right) \, ,
\end{equation}
we have the equation
\begin{equation}\label{eq: eL_in _phis}
\left\lbrace
    \begin{aligned}
        \partial_t u - \Delta u &= - \mathbb{P} (\bar u \cdot \nabla u + u\cdot \nabla \bar u) && \text{ in } \R^3 \times (0,e^2-1)  \\
        u(x,0) &= U_0(x) && \text{ in } \R^3 \, .
    \end{aligned}
    \right.
\end{equation}
 Using that $\bar U$ is smooth with compact support, we can easily prove that
\begin{equation}
    \| u(\cdot, t) \|_{L^2} + t^{1/2}\| \nabla u(\cdot, t) \|_{L^2}
    \le C(\bar U) \| U_0 \|_{L^2}\, ,
    \qquad t \in (0, 10) \, ,
\end{equation}
which gives
\begin{equation}
    \| U(\cdot, \tau) \|_{L^2} + \tau^{1/2}\| \nabla U(\cdot, \tau) \|_{L^2}
    \le C(\bar U) \| U_0 \|_{L^2}\, ,
    \qquad \tau \in (0, 2) \, .
\end{equation}
The latter implies \eqref{eq:reg1} for $k=1$ and $m=0,1$. The general case follows by induction studying the equation solved by $\nabla^k u$ which has a structure similar to \eqref{eq: eL_in _phis} but with forcing and additional lower order terms.

{\bf Step 2}. For any $\delta>0$ it holds
\begin{equation}\label{eq:reg2}
		\| e^{ \tau \L_{\rm ss}} U_0 \|_{H^k} \le M(k,\delta) e^{\tau(a+\delta)}
		\| U_0 \|_{L^2} \, ,
		\quad \forall\, U_0\in L^2_\sigma, \, \, \tau \ge 2 \, .
	\end{equation}
Using the semigroup property, Step 1 with $m=0$, and \eqref{eq: growth est}, we have
\begin{equation}
  \begin{split}
            \| e^{\tau \L_{\rm ss}} U_0 \|_{H^k}  & = \| e^{\kappa \L_{\rm ss}}(e^{(\tau - \kappa)\L_{\rm ss}} U_0 ) \|_{H^k}
    \\ &
    \le M(k) \kappa^{-k/2} \| e^{(\tau - \kappa)\L_{\rm ss}} U_0  \|_{L^2}
    \\& \le M(k,\delta) \kappa^{-k/2}e^{(\tau - \kappa)(a+\delta)} \| U_0  \|_{L^2} \, .
  \end{split}
\end{equation}
The claimed inequality \eqref{eq:reg2} follows by choosing $\kappa =1$. 

It is immediate to see that the combination of~\eqref{eq:reg1} in Step~1 and~\eqref{eq:reg2} in Step~2 gives \eqref{eq:reg} for integers $\sigma_2 \geq \sigma_1 \geq 0$. This completes the proof.
\end{proof}

As a consequence of Proposition~\ref{prop:specLOmega} and Proposition~\ref{prop:semigroup}, it holds
\begin{equation}\label{eq: Ulin H1}
	\| U^{\rm lin}\|_{H^k} = C(k,\eta) e^{a \tau} \, \quad \text{for any $\tau \ge 0$ and $k\in \mathbb{N}$} \, .
\end{equation}

\subsection{Nonlinear construction}
\label{sec:nonlinearconstruction}

In this section, we solve the nonlinear problem
\begin{equation}\label{eq:nonlin2}
\begin{aligned}
	&\partial_\tau U^{\rm per} - \L_{\rm ss} U^{\rm per} 
	=
	- \mathbb{P}\left[ (U^{\rm per} \cdot \nabla) U^{\rm per} \right] \\
	&\quad -  \bP \left[ ( U^{\rm lin} \cdot \nabla) U^{\rm per}  +  (U^{\rm per} \cdot \nabla)  U^{\rm lin} \right] - \bP \left[ (U^{\rm lin} \cdot \nabla) U^{\rm lin} \right] \, .
	\end{aligned}
\end{equation}

\begin{proposition}\label{thm:nonlin}
	    Assume $a=s(\L_{\rm ss})>0$ and $N> 5/2$ is an integer. Then there exist $T=T(\bar U,U^{\rm lin},N) \in \R$, $\eps_0=\eps_0(a)>0$ and $U^{\rm per} \in C((-\infty, T]; H^N(\R^3;\R^3))$, a solution to \eqref{eq:nonlin2}, such that
		\begin{equation}
			\| U^{\rm per}(\cdot, \tau )\|_{H^N} 
			\le  e^{(a + \eps_0) \tau } \, ,
			\quad
			\text{for any $\tau \leq T$} \, .
		\end{equation}
\end{proposition}

Once Proposition~\ref{thm:nonlin} is known, the solution $U^{\rm per}$ can be bootstrapped to satisfy the refined decay estimates
\begin{equation}
     \| U^{\rm per} \|_{H^k} \les_k e^{2a\tau}, \quad \tau \in (-\infty,T], \; k \geq 0 \, ,
\end{equation}
thereby completing the proofs of Theorem~\ref{thm:nonuniqueness criterion} and Theorem~\ref{thm:refined}. This bootstrapping is discussed in Section~\ref{sec:refineddecay}. 


The remaining part of this section is devoted to the proof of Proposition~\ref{thm:nonlin} by means of a fixed point argument.


\begin{remark}
In principle, one could infer the existence, uniqueness, and smoothness of the unstable manifold from abstract theory for semilinear parabolic PDEs, see~\cite{henrybook}. It is simple and instructive to give a construction below, which can be generalized to non-compactly supported backgrounds, see Remark~\ref{rmk:generalbackgrounds}.
\end{remark}

\subsubsection{Functional setting}
From now on, as in Proposition~\ref{thm:nonlin}, we suppose $a = s(\L_{\rm ss}) > 0$. Let $N > 5/2$ be an integer. It will be enough to choose $\eps_0 = a/2$. Let $T \in \R$,  to be determined. We suppress dependence of the constants below on $\bar{U}$ and $U^{\rm lin}$, which are now fixed.

We consider the norm
\begin{equation}
	\|U\|_{X}:= \sup_{\tau < T} e^{-(a + \eps_0) \tau} \|U(\cdot,\tau)\|_{H^N} \, ,
\end{equation}
and the associated Banach space
\begin{equation}
	X:= \{  U \in C((-\infty, T]; H^N(\R^3;\R^3)) \, : \, \|U\|_{X} < \infty  \} \, .
\end{equation}
We study the functional 
\begin{align*}
	\mathcal{T}(U)(\cdot, \tau) 
	& = -\int_{-\infty}^{\tau} e^{(\tau - s) \L_{\rm ss}} \circ  \mathbb{P} \left[ ( U \cdot \nabla) U 
	+ ( U^{\rm lin} \cdot \nabla) U  +  (U \cdot \nabla)  U^{\rm lin} \right] ds
	\\ &\qquad\qquad  - \int_{-\infty}^{\tau} e^{(\tau - s) \L_{\rm ss}} \circ \mathbb{P} \left[ (U^{\rm lin}\cdot \nabla) U^{\rm lin} \right] ds   	
	\, .
\end{align*}
By Duhamel's formula and parabolic regularity theory, any $U\in X$ such that $\mathcal{T}(U) = U$ is a solution to \eqref{eq:nonlin2} satisfying the statement of Proposition~\ref{thm:nonlin}.

To find the sought fixed point we apply the contraction mapping principle. 
\begin{proposition}\label{prop: fixed point}
	 There exist $T=T(\bar U,U^{\rm lin},N) \in \R$ and $\eps_0=\eps_0(a)$ such that
	\begin{equation}
		\mathcal{T} : \{ U\in X: \|U\|_X\le 1\} \to \{ U\in X: \|U\|_X\le 1\} \, ,
	\end{equation}
    is a contraction.
\end{proposition}

\begin{proof}

We first prove a claim which, under suitable assumptions, builds a fixed point of an operator with our structure.

\subsubsection{Abstract claim} {\it Let $X$ be a Banach space and
	\begin{equation}
		\mathcal{T}(U) := B(U,U) + LU + G \, ,
		\quad U\in X 	\, ,
	\end{equation}
	where $G\in X$, $L : X\to X$ is a bounded linear operator and $B: X \times X \to X$ is a bounded bilinear form. If
	\begin{equation}
		\| L\| + 2\|B\| + \|G\|_X < 1 \, ,
	\end{equation}
	then 
	\begin{equation}
		\mathcal{T} : \{ U\in X: \|U\|_X\le 1\} \to \{ U\in X: \|U\|_X\le 1\} \, ,
	\end{equation}
	is a contraction.	
	}

Indeed, if $\| U \|_X\le 1$, then
\begin{equation}
    \| \mathcal{T} U\|_X 
    \le \| B(U,U)\|_X + \| L U \|_X + \| G \|_X
    \le \| B \| \|U\|_X^2 + \| L \| \| U\|_X + \| G\|_X < 1 \, .
\end{equation}
Let us now consider $U, V \in X$ such that $\| U\|_X \le 1$, $\|V\|_X\le 1$. It holds
\begin{align}
    \| \mathcal{T}(U) - \mathcal{T}(V) \|_X 
    & \le \| B(U,U) - B(V,V)\|_X + \| L(U-V)\|_X
    \\& \le \| B(U-V,U)\|_X + \| B(V, U-V)\|_X + \|L\| \|U-V\|_X
    \\&\le  (2\| B \| + \| L \|) \|U-V\|_X \, ,
\end{align}
since $2\| B \| + \| L \| < 1$, we conclude that $\mathcal{T}$ is a contraction.

\subsubsection{Application of the claim to Proposition~\ref{prop: fixed point}}
We now turn to verify the assumptions of the previous elementary claim for our operators:
\begin{equation}
	B(U,V) (\cdot, \tau) := -\int_{-\infty}^{\tau} e^{(\tau - s) \L_{\rm ss}} \circ  \mathbb{P} \left[ ( U \cdot \nabla) V \right]  ds \, ,
\end{equation}
\begin{equation}
	LU(\cdot, t) := -\int_{-\infty}^{\tau} e^{(\tau - s) \L_{\rm ss}} \circ  \mathbb{P} \left[ 
	( U^{\rm lin} \cdot \nabla) U  + (U \cdot \nabla)  U^{\rm lin} \right] ds \, ,
\end{equation}
\begin{equation}
	G(\cdot, \tau) := -\int_{-\infty}^{\tau} e^{(\tau - s) \L_{\rm ss}} \circ \mathbb{P} \left[  (U^{\rm lin}\cdot \nabla) U^{\rm lin} \right] ds  \, .
\end{equation}

More precisely, we prove below that 
\begin{align}\label{eqn:verify}
	\| B(U,U)\|_X + \|LU\|_X + \|G\|_X
	\le C(N,\delta, a) e^{T(a-\eps_0)} \, ,
\end{align}
provided $\delta < a$. Choosing $\delta = \eps_0/2 = a/4$ and $T$ 
negative enough we conclude that the right-hand side in \eqref{eqn:verify} is less than $ 1/2$. Applying the previous abstract claim, the proof of  Proposition~\ref{prop: fixed point} is concluded provided we show \eqref{eqn:verify}.

To this end, we will use the following observation: For $N > 5/2$, the space $H^{N-1}(\R^3)$ is an algebra, and hence for every $f,g \in H^N(\R^3)$,
\begin{equation}\label{lemma: HN}
\| f \nabla g \|_{H^{N-1}}
\le C(N) \| f \|_{H^{N-1}} \| \nabla g \|_{H^{N-1}}
\le C(N) \| f \|_{H^N} \| g \|_{H^N} \,.
\end{equation}


\subsubsection{Estimate on $B(U,U)$}
We show that
\begin{equation}\label{eq: B estimate}
	\|B(U,U)\|_X \le C(N,\delta,d) e^{(a+\eps_0)T} \|U\|_X^2 \, .
\end{equation}
We apply Lemma~\ref{prop:semigroup} to get
\begin{equation}
	\norm{B(U,U)(\cdot, \tau)}_{H^{N+1/2}} \le M(N,\delta) \int_{-\infty}^\tau \frac{e^{(\tau-s)(a+\delta)}}{(\tau-s)^{3/4}} \norm{(U \cdot \nabla) U(\cdot, s)}_{H^{N-1}}  ds \, ,
\end{equation}
for any $\tau \in (-\infty, T)$.
Thanks to \eqref{lemma: HN}, we deduce
\begin{equation}
	\norm{(U \cdot \nabla) U(\cdot, s)}_{H^{N-1}}
	\le C(N) \norm{U(\cdot, s)}_{H^N}^2
	\le C(N) e^{2(a+\eps_0)s} \norm{U}_X^2 \, ,
\end{equation}
and, using that $a+2\eps_0-\delta>0$, we obtain
\begin{align}
	\norm{B(U,U)(\cdot, \tau)}_{H^{N+1/2}} 
	& \le C(N,\delta) \norm{U}_X^2 \int_{-\infty}^{\tau} \frac{e^{(\tau-s)(a+\delta)}e^{2(a+\eps_0)s}}{(\tau-s)^{3/4}}  d s
	\\& \le C(N,\delta) e^{2(a+\eps_0)\tau} \norm{U}_X^2 \, ,
\end{align}
which implies \eqref{eq: B estimate}. The additional gain of $1/2$ regularity is not necessary for the existence; we use it to bootstrap the solution to smoothness in Section~\ref{sec:refineddecay}.

\subsubsection{Estimate on $G$}\label{sec:estimate G}
We show that
\begin{equation}\label{eq: G estimate}
	\| G\|_X \le C(N,\delta,a)e^{ T(a-\eps_0)} \, ,
\end{equation}
provided $\delta < a$.

As a consequence of Lemma~\ref{prop:semigroup} and \eqref{lemma: HN}, we have
\begin{equation}
	\| G(\cdot, \tau)\|_{H^N} \le M(N,\delta) \int_{-\infty}^\tau e^{(\tau-s)(a+\delta)}
	\| U^{\rm lin} \cdot \nabla  U^{\rm lin}\|_{H^{N}}
	  ds \, ,
\end{equation}
\begin{equation}
	\| U^{\rm lin} \cdot \nabla  U^{\rm lin}\|_{H^{N}}
	\le 
	C(N) \| U^{\rm lin} \|_{H^{N+1}}^2
	\le C(N) e^{2as} \, .
\end{equation}
Hence, 
\begin{align}
\label{eq:goodgestimate}
	\| G(\cdot, \tau)\|_{H^N} \le C(N,\delta)  \int_{-\infty}^\tau e^{(\tau-s)(a+\delta)}e^{2sa}
	d s  \le
	C(N,\delta,a)e^{2\tau a} \, ,
\end{align}
provided $\delta < a$.
This implies \eqref{eq: G estimate}.

\subsubsection{Estimate on $LU$}
We show that
\begin{equation}\label{eq: L estimate}
	\|L U\|_X
	 \le C(N,\delta,a) e^{a T} \|U\|_X\, .
\end{equation}
As a consequence of Lemma~\ref{prop:semigroup} we get
\begin{align}
	\| LU & (\cdot, \tau)\|_{H^{N+1/2}} 
	\\& 
	\le M(N,\delta)  \int_{-\infty}^\tau \frac{e^{(\tau-s)(a+\delta)}}{(\tau-s)^{3/4}}
	(    
	 \|(U^{\rm lin} \cdot \nabla) U\|_{H^{N-1}}
	 +\|(U \cdot \nabla) U^{\rm lin}\|_{H^{N-1}}) d s \, ,
\end{align}
for $\tau \in (-\infty, T)$.

By employing \eqref{eq: Ulin H1} and \eqref{lemma: HN} we deduce
\begin{equation}
	\| (U^{\rm lin} \cdot \nabla) U\|_{H^{N-1}}
	+\|(U \cdot \nabla) U^{\rm lin}\|_{H^{N-1}}
	\le C(N,\delta) e^{(2a+\eps_0)s} \|U\|_X \, .
\end{equation}
This implies
\begin{equation}
		\|LU(\cdot, \tau)\|_{H^{N+1/2}} \le  C(N,\delta,a) e^{(2a+\eps_0)s} \|U\|_X\, ,
\end{equation}
and \eqref{eq: L estimate} easily follows.

Collecting the estimates \eqref{eq: B estimate}, \eqref{eq: G estimate}, \eqref{eq: L estimate}, we obtain \eqref{eqn:verify}. \end{proof}


\subsubsection{Refined decay estimates}
\label{sec:refineddecay}
Once the solution $U^{\rm per}$ of Proposition~\ref{thm:nonlin} is known to exist, we can bootstrap it to satisfy $O(e^{2\tau a})$ decay in all $H^k$, $k \geq 0$, as claimed in Theorem~\ref{thm:refined}. First, the estimate~\eqref{eq:goodgestimate} on $G$ is already satisfied for all $N$. Second, the estimates on $B(U^{\rm per},U^{\rm per})$ and $LU^{\rm per}$ above gain $1/2$ derivative above the regularity of the solution. By Duhamel's formula, we see that if $N > 5/2$ (not necessarily integral) and
\begin{equation}
\sup_{\tau \in (-\infty,T]} e^{-(a+\varepsilon_0)\tau} \| U^{\rm per}(\cdot,\tau) \|_{H^N} < +\infty \, ,
\end{equation}
 then
 \begin{equation}
    \label{eq:inductionconclusion}
     \sup_{\tau \in (-\infty,T]} e^{-2 a \tau} \| U^{\rm per}(\cdot,\tau) \|_{H^{N+1/2}} < +\infty \, .
 \end{equation}
 By induction, we conclude that~\eqref{eq:inductionconclusion} holds for all $N > 5/2$.



\begin{remark}[General backgrounds]
\label{rmk:generalbackgrounds}
It is possible to treat more general unstable backgrounds
 satisfying the natural decay conditions
\begin{equation}
    \label{eq:conditiononthebackground}
    |\nabla^k \bar{U}(\xi)| \les_k \la \xi \ra^{-k-1}
\end{equation}
and $\bar{u}(x,t) \to \bar{u}_0$ in $L^2_{\rm loc}(\R^3)$ as $t \to 0^+$.
Let $0 \leq \chi \in C^\infty(B_2)$ with $\chi \equiv 1$ on $B_1$. For $\alpha \in (0,1/2)$, we define
\begin{equation}
    \tilde{u}(x,t) = \bar{u}(x) \chi\left( \frac{x}{t^{\alpha}} \right).
\end{equation}
That is, the cut-off shrinks as $t \to 0^+$ but `more slowly' than the self-similar rate (in similarity variables, the tails of the cut-off are escaping to $|\xi| = +\infty$ as $\tau \to -\infty$). The ansatz is instead $u = \tilde{u} + u^{\rm lin} + u^{\rm per}$. The cut-off introduces additional forcing and lower-order terms, which are treated perturbatively. (With this forcing, the refined decay $O(e^{2\tau a})$ is not expected, but it is not necessary for non-uniqueness.) 
\end{remark}

\subsubsection*{Acknowledgments} DA was supported by NSF Postdoctoral Fellowship  Grant No.\ 2002023 and Simons Foundation Grant No.\ 816048.
EB was supported by Giorgio and Elena Petronio Fellowship.
MC was supported by the SNSF Grant 182565.
The authors thank Camillo De Lellis, Vikram Giri, Maximilian Janisch, and Hyunju Kwon for running together a reading seminar on the results by Vishik; 
in particular, they are grateful to Camillo De Lellis for sharing his original point of view on them and for many useful discussions. DA also thanks Vlad Vicol, Matt Novack, and Vladim{\'i}r {\v S}ver{\'a}k for discussions and encouragement and Tarek Elgindi for comments on the relationship between non-uniqueness and blow-up above Section~\ref{sec:comparisonwithexisting}.

\bibliographystyle{abbrv}
\bibliography{nonuniquenessbib}

\end{document}